\documentclass[11pt]{amsart}
\usepackage[utf8]{inputenc}     % pour l'encodage des caractères spéciaux
\usepackage[T1]{fontenc}    % encodage des polices 
\usepackage{lmodern}    % police de caractères Computer modern
\usepackage[english]{babel}  % gestion de plusieurs langues
\usepackage{csquotes}   % gestion des guillemets

\usepackage[dvipsnames]{xcolor}     % colorer le texte
\usepackage{comment}
\usepackage[a4paper]{geometry}   % modifier les marges

\usepackage{amsmath}     % environnement et commandes
\usepackage{amssymb}    % symboles
\usepackage{amsthm}      % environnements de théorèmes
\usepackage{mathtools}   % autres symboles

\makeatletter
\def\@tocline#1#2#3#4#5#6#7{\relax
	\ifnum #1>\c@tocdepth % then omit
	\else
	\par \addpenalty\@secpenalty\addvspace{#2}%
	\begingroup \hyphenpenalty\@M
	\@ifempty{#4}{%
		\@tempdima\csname r@tocindent\number#1\endcsname\relax
	}{%
		\@tempdima#4\relax
	}%
	\parindent\z@ \leftskip#3\relax \advance\leftskip\@tempdima\relax
	\rightskip\@pnumwidth plus4em \parfillskip-\@pnumwidth
	#5\leavevmode\hskip-\@tempdima
	\ifcase #1
	\or\or \hskip 1em \or \hskip 2em \else \hskip 3em \fi%
	#6\nobreak\relax
	\dotfill\hbox to\@pnumwidth{\@tocpagenum{#7}}\par
	\nobreak
	\endgroup
	\fi}
\makeatother
\theoremstyle{plain}    % pour les thm
\newtheorem{Thm}{Theorem}[subsection] % numéroté par section
\newtheorem{Lem}[Thm]{Lemma} % numéroté comme thm
\newtheorem{Prop}[Thm]{Proposition} % numéroté comme thm
\newtheorem{Cor}[Thm]{Corollary}

\theoremstyle{definition}
\newtheorem{Def}[Thm]{Definition} % numéroté comme thm
\newtheorem{Ex}[Thm]{Example}
\newtheorem{Not}[Thm]{Notation}

\theoremstyle{remark}
\newtheorem{Rq}[Thm]{Remark}

\newtheorem*{Conv}{Convention}

\newtheoremstyle{component}{}{}{}{}{\itshape}{.}{.5em}{\thmnote{#3}#1}
\theoremstyle{component}
\usepackage{enumitem} % Changer comportement des listes 

\usepackage[noadjust]{cite} 

\usepackage[
unicode=true,
pdfstartview=FitV,
colorlinks=true,
citecolor=Orange,
linkcolor=NavyBlue,
urlcolor=BlueViolet,
linktoc=page,
hyperindex=true,
pdfcreator={},
]{hyperref} % liens hypertexte

\usepackage{graphicx} % to import images
\graphicspath{{./Images/}}
\usepackage[hypcap=false]{caption}  % légender figures
\usepackage{subcaption}

\usepackage{tikz}   
\usepackage{tikz-cd}

\usetikzlibrary{arrows}

\tikzstyle{stuff_fill}=[rectangle,fill=white,minimum size=1em]

\tikzstyle{stuff_fillc}=[circle,fill=white,minimum size=1em]

\tikzstyle{stuff_fillg}=[rectangle,fill=black!10,minimum size=1em]

\tikzstyle{stuff_nofill}=[rectangle,draw]

\newcommand{\N}{\mathbf{N}} %entiers
\newcommand{\Z}{\mathbf{Z}} %entiers relatifs
\newcommand{\Q}{\mathbf{Q}} %rationnels
\newcommand{\R}{\mathbf{R}} %réels
\newcommand{\A}{\mathcal{A}} 
\renewcommand{\H}{\mathbf{H}} %espace hyperbolique
\renewcommand{\S}{\mathbf{S}}
\newcommand{\Sc}{\mathcal{S}}
\newcommand{\1}{\mathbf{1}}

\DeclareMathOperator{\Ad}{Ad}
\DeclareMathOperator{\ad}{ad}
\DeclareMathOperator{\Diag}{Diag}
\DeclareMathOperator{\Ho}{H} %Homologie

\newcommand{\defeq}{\vcentcolon=} %définition

\newcommand{\length}{|\!\cdot\!|}

\newcommand{\sg}[1][t]{\langle #1\rangle} %sous-groupe engendré par #1
\newcommand{\lb}{[} % Pour commutateurs de groupes : (\!(
\newcommand{\rb}{]} %

\newcommand{\lev}{\{} % Pour évaluer les mots en chemins : évaluation
\newcommand{\rev}[1][G]{\}_{#1}}
\DeclareMathOperator{\ev}{ev}

\newcommand{\overbar}[1]{\mkern 1.5mu\overline{\mkern-1.5mu#1\mkern-1.5mu}\mkern 1.5mu}

\newcommand{\n}{\mathfrak{n}}

\DeclareMathOperator{\Cone}{Cone}
\DeclareMathOperator{\Precone}{Precone}

\DeclareMathOperator{\SOLV}{SOL} %Groupe SOL
\newcommand{\SOL}[1][\alpha]{\SOLV_{#1}} %Sol_{\alpha} de poids 1,-\alpha
\newcommand{\Osc}[1][\alpha]{\mathrm{Osc}_{#1}}
\newcommand{\Heis}{\mathrm{Heis}}

\newcommand{\cone}[1][\SOLV_{\alpha}]{\Cone_{\omega}(#1)} % Cône de SOL_alpha
\newcommand{\Rexp}[1][G]{\mathrm{R}_{\exp} #1}

\newcommand{\Earr}{\mathrm{Earr}} %Boucles d'oreilles hawaiennes
\newcommand{\SqEarr}{{\mathrm{SqEarr}}} % Boucles d'oreilles carrées
\newcommand{\Cocoa}[1][k]{\mathcal{C}_{#1}} %Cabosse de cacao à $k+1$ branches

\hypersetup{ %Propriétés du PDF
	pdfauthor={Antoine Velut},
	pdftitle={Fundamental groups of asymptotic cones of Lie groups with the SOL obstruction},
	pdfsubject={},
	pdfkeywords={Asymptotic cones, Lie groups, Quasi-isometries, Dehn function, 
	Fundamental group, Hawaiian earring}
}

\title[Fundamental groups of asymptotic cones]{Fundamental groups of asymptotic cones of Lie groups with the SOL obstruction}
\author{Antoine Velut}
\address{Unité de Mathématiques Pures et Appliquées, UMR 5669 CNRS, 
	École normale supérieure de Lyon, 46 allée d’Italie, 69364 Lyon cedex 07, France}
\email{antoine.velut@ens-lyon.fr}
\date{\today}

\begin{document}

\begin{abstract}
	We study asymptotic cones of Lie groups, presenting a link between the geometry of the weights 
	and the highly non-simply-connected nature of asymptotic cones. We show that a Lie group 
	that admits a group of SOL-type as a quotient
	is such that the fundamental group of its asymptotic cone contains 
	the fundamental group of the Hawaiian earring space. We also obtain a strong non-vanishing 
	result for the first homology group of the asymptotic cone. 
	
	The assumption on the Lie group, introduced by Abels and called ``the SOL obstruction'', is 
	equivalent to an explicit geometric condition on the weights.
	Our work builds upon the results of Burillo, who proves the statement on asymptotic cones in 
	the case of SOL, and on those of Cornulier and Tessera who show that the SOL obstruction 
	implies exponential growth of the Dehn function.	
\end{abstract}

\maketitle

\tableofcontents

\section{Introduction}

\subsection{Motivations}

This paper deals with the large-scale geometry of Lie groups. More precisely, we are interested in computing quasi-isometry invariants of Lie groups. The quasi-isometry type plays here the role of homeomorphisms in topology, or diffeomorphisms in differential geometry.

We study one large-scale invariant in particular, the \emph{asymptotic cone}. Informally speaking, 
it is the metric space obtained when looking at a given metric space ``from infinitely far away''. 
It is bi-Lipschitz invariant under quasi-isometry, so one may use its usual topological invariants 
as large-scale invariants for the metric space being studied. More precisely, one associates to a 
metric space a collection of asymptotic cones which depend on a technical choice (of a 
non-principal \emph{ultrafilter}), to then study properties independent of this choice.

We will be mostly interested in the fundamental group of asymptotic cones. This has already 
proven to be a fruitful object of study, linked with another classical quasi-isometry invariant, 
the \emph{Dehn function}. In particular, having all asymptotic cones simply connected is a natural 
geometric property, in view of the following delightful interpretation not relying on asymptotic 
cones.

\begin{Thm}[{\cite{gromov1993asymptotic}}, see {\cite[Section 4]{drutu2002quasi-isometry}}] \label{thm:cone_dehn}
	Let $X$ be a geodesic metric space.  All asymptotic cones of $X$ are simply connected if and 
	only if $X$ satisfies the ``loop division property'', meaning roughly that every loop 
	in $X$ can be divided into a uniformly bounded number of loops of half its length.
	
	Whenever this holds, the Dehn function of $X$ grows at most polynomially. 
\end{Thm}

We study here a class of groups that is known to have exponential Dehn function, namely groups with 
the \emph{SOL obstruction}. For real Lie groups, it is the property of having a group $\SOL$ as a 
quotient for some $\alpha>0$, where $\SOL=\R^2\rtimes_{(1,-\alpha)} \R$ is the semidirect product 
with law given by 
\[
(x,y,t).(x',y',t') = (x+e^{t}x', y+e^{-\alpha t}y', t+t').
\]
The unimodular group $\SOL[1]$ is more commonly called $\SOLV$.

The SOL obstruction was first introduced by Abels in \cite{abels1987finite}, using an equivalent 
condition involving weigths of the Lie algebra, as an obstruction for 
$p$-adic Lie groups to be compactly presented. In \cite{cornulier2017geometric}, 
Cornulier and Tessera restate this condition as having $\SOL$ as a quotient, and show that 
for real Lie groups (which are always compactly presented) this property implies exponential growth 
of the Dehn function. 

This implies, in general,  
non-simple-connectedness of \emph{at least one} asymptotic cone. However, this does not apply to 
\emph{all} asymptotic cones, and it does not give any more information on the fundamental groups of 
the asymptotic cones. For example, the results of \cite{cornulier2017geometric} do not tell whether 
these fundamental  groups are abelian or not.
\smallskip

In the case of the classical group $\SOLV$, Burillo shows in \cite{burillo1999dimension}  that the 
fundamental group of the asymptotic cone 
is ``highly non-trivial'', meaning here that it contains the fundamental group of the Hawaiian earring. In particular, it is neither abelian nor free.

\subsection{Main result}
To state our results, we need two ingredients: groups to study, and topological spaces to exhibit large fundamental groups.

First, the topological space. We denote by $\Earr$ the \emph{Hawaiian earring space}, meaning the 
union of countably many circles that are all tangent at the same point and whose radii tend to 
zero. It can also be constructed as the one-point compactification of the product $\R\times \N$. 
This 
space will be our source of pathological fundamental groups. Indeed, its fundamental group 
$\pi_1(\Earr)$
is uncountable, contains a free group over countably many generators, and is not free; 
see Section \ref{subsec:hawaiian_earring} for statements and references. 
\smallskip

Let us present now the class of groups we study.
The following definition appears for instance in \cite[Lemma 2.5]{cornulier_dimension_2008}.

\begin{Def}
	A \emph{real triangulable group} is a Lie group isomorphic to a closed connected group of real triangular matrices. 
	Equivalently, it is a simply connected solvable group 
	such that the spectrum of the adjoint action of every element is purely real.
\end{Def}

For a proof of the equivalence, see \cite[Lemma 3.1]{conversano2018solvable}\footnote{This reference proves the equivalence in the case of Lie algebras. To prove this for Lie groups, use the following two ingredients. First, the simple connectedness of the group which implies that Lie algebra morphisms integrate at the level of groups. Second, the relationship $\exp \circ \ad= \Ad \circ \exp_G $ which shows that the condition that the adjoint spectrum is purely real is the same at the group and at the algebra level.}. For example, the group $\SOL$ is real triangulable.

Considering only this class of Lie groups is not really a restriction, since every connected Lie group is quasi-isometric to a real triangulable group; see for instance \cite[Lemma 3.A.1]{cornulier2017geometric}.

Our main result, to which this paper is devoted, is the following. 
\begin{Thm} \label{thm:main}
	Let $G$ be a real triangulable Lie group. Assume $G$ has the \emph{SOL obstruction}, meaning it admits $\SOL$ as a quotient for some $\alpha>0$. 
	Then for every non-principal ultrafilter $\omega$, the fundamental group of $\cone[G]$ contains 
	the group $\pi_1(\Earr)$. In particular, it is uncountable, nonabelian and nonfree.
\end{Thm}

Many groups admit the SOL obstruction, and for real triangulable Lie groups it is equivalent to an 
explicit condition on the weights; see Example \ref{ex:groups_SOL_obstruction}.
The above theorem is a generalization of Corollary 22 in \cite{burillo1999dimension}, which states 
this result in the case $G=\SOLV$. 
In this case, Burillo considers large loops representing commutators. In our general setting, we 
need to modify these because of the more complicated algebraic structure, by, roughly speaking, 
considering suitable iterated commutators.
To achieve this, we show that the combinatorial loops defined in \cite[Section 
12]{cornulier2017geometric}, which lift to groups with the SOL obstruction, provide subspaces 
homotopically equivalent to wedges of circles inside 
the asymptotic cone of $\SOL$. 
Another difficulty not encountered in the case $G=\SOLV$ is that the loops need not be embeddings 
themselves, hence the need to study the combinatorics of these iterated commutators to ensure that 
they remain homotopically injective.
This allows us to produce a nontrivial loop in the asymptotic cone.
We then produce similar loops at smaller and smaller scales, 
so as to produce Hawaiian earring spaces inside these cones.

Moreover, we obtain a similar result of strong non-vanishing for homology.

\begin{Thm} 
	\label{thm:main_homology}
	Let $G$ be a real triangulable Lie group with the SOL obstruction. 
	Then for every non-principal ultrafilter $\omega$, the first homology group 
	$\Ho_1(\cone[G])$ 
	contains the Baer-Specker group $\Z^\N$. 
	In particular, $\Ho_1(\cone[G],\Q)$ has continuum dimension over $\Q$.
\end{Thm}

The case $G=\SOL$, while not explicitly stated, can be deduced from \cite{burillo1999dimension}. 
The general SOL obstruction case, however, is new.

\begin{Rq}
	Note that it is shown in \cite[Section 12]{cornulier2017geometric} that the loops used in this paper have exponential area. This is the key to one of the results of Cornulier and Tessera which states that groups with the SOL obstruction have an exponential Dehn function. 
	
	Thus Theorem \ref{thm:main} establishes yet another link between exponential growth of the area and non-contractibility in the asymptotic cone.
	Indeed, recall that by Theorem \ref{thm:cone_dehn}, a connected Lie group whose asymptotic 
	cones are simply connected has a polynomial Dehn function. 
	In view of the results of \cite{cornulier2017geometric}, every group with the SOL obstruction 
	has some asymptotic cone which is not simply connected. Our result can be seen as an 
	improvement of this statement for two reasons. 
	First, it holds for all ultrafilters. 
	And second, the fundamental group is not only nontrivial but even ``large'', in the sense that 
	it is uncountable and contains a nonabelian free group.
	
	\begin{Rq}
		\label{rq:obstruction_homologique}
		Beware that exponential growth of the Dehn function does not always imply the existence of 
		nonabelian free subgroups in the fundamental group of the asymptotic cone. 
		Hence Theorem \ref{thm:main} is not a direct 
		consequence of \cite[Section 12]{cornulier2017geometric}.
		
		Indeed, in \cite{abels1987finite}, Abels introduces not only the SOL obstruction, but also 
		the \emph{homological obstruction}. It is proven in \cite{cornulier2017geometric} that for 
		real Lie groups, it also implies exponential growth of the Dehn function. 
		Hence groups with the homological obstruction have non-simply-connected cones as well. 
		However, 
		a quotient of Abel's first group, considered in \cite[Cor. 
		1.4]{cornulier2013dehn}, has 
		the homological obstruction but is such that the fundamental group of its asymptotic cone 
		is nontrivial abelian.
	\end{Rq}
\end{Rq}

\subsection{Sketch of proof}

Let us give an overview of the proof of Theorem \ref{thm:main}. 

We are given a real triangulable Lie group $G$ and, for some $\alpha>0$, 
a surjective homomorphism $p:G \to \SOL$. This induces a continuous surjection 
\[
	p_{\omega} : \cone[G] \to \cone
\] 
between asymptotic cones. 

Let $\lb u,v \rb =uvu^{-1}v^{-1}$ denote the group commutator. Roughly speaking, the loops 
considered by Burillo in $\SOLV$ have the form $\lb t^nxt^{-n},y \rb$, where $x,y,t$ are the basis 
vectors. These words are usually not relations in $G$, having in mind that the exponential radical 
might fail to be abelian (see Section \ref{subsec:exp_rad}). 
So we use $k$-fold iterated commutators of the 
form $\lb \dots \lb t^nxt^{-n},y\rb , \dots,y]$
as introduced in \cite[Section 12]{cornulier2017geometric}. 
Indeed, for some large enough $k$, fixed and depending only on $G$, 
these are relations in $G$. 

Following Cornulier-Tessera, we take a combinatorial point of view
and define loops using the words appearing in iterated commutators, defined in Section 
\ref{sec:setup}, evaluated on 
suitable group elements. This framework allows for a clear distinction between geometric and 
combinatorial arguments. So we start in Section \ref{sec:combinatorial} by gathering some 
combinatorial properties of these words.

We then evaluate these words into loops in Section \ref{sec:geom_loops}. 
We build a loop inside the asymptotic cone 
$\cone$ by considering a sequence of loops of linear length in $\SOL$ and 
taking the ultralimit. 
We then want to ensure that the lifts of these loops to $G$ remain loops, so in Section 
\ref{sec:higher_order_loops} we use the nilpotency 
of the exponential radical of $G$ and iterate commutators up to a fixed rank.

The difficulty then is that there is no classification of general solvable Lie groups, 
so it is not obvious to study the geometry of such loops in $G$. 
This is where the quotient group $\SOL$ is of use, 
as we are able to describe the metric properties of the base loops inside its asymptotic cone. 

While iterating commutators is crucial for the lifts to still be loops, it creates overlaps inside 
the base loops, preventing them from being embedded inside $\cone$. However, careful 
computation shows that the image of each of the base loops is a space that we call 
``cocoa pod'', which is homotopically equivalent to a finite wedge of circles.
We therefore need an extra step of factorization through the cocoa 
pod of the maps at the $\SOL$-level, 
in order to decompose them into a map that contains all the overlap and an actual bi-Lipschitz 
embedding into $\cone$. The key is that the overlap map sends generators of the fundamental group 
to \emph{positive} words over a well-chosen alphabet, so it remains $\pi_1$-injective.

Once this is done, in Section \ref{sec:change_scale} we iterate this construction with sequences of 
loops of length approximately $n, 
\frac{n}{2},\frac{n}{3},\dots$ and so on. This gives rise to a sequence of loops in $\cone[G]$ 
whose lengths tend to zero, yielding a map $\Phi : \Earr \to \cone$ that lifts to a map
$\widetilde{\Phi} :\Earr\to \cone[G]$ satisfying $\Phi = p_{\omega} \circ \widetilde{\Phi}$.

Again, the change-of-scale creates overlaps, and we show that the subspace drawn by 
the loops in the asymptotic cone is a ``cocoa pod of square earrings'', which is homotopically 
equivalent to the Hawaiian earrings.
So we factor the map $\Phi$ from $\Earr$ to $\cone$ 
through this metric space to regain an embedding. 
Here, showing that the map containing all the overlap is $\pi_1$-injective uses again the
positivity of the image of generators, as well as an explicit 
description of $\pi_1(\Earr)$ as a subgroup of an inverse limit of free groups.

To put the final nail in the coffin, we apply two results from \cite{burillo1999dimension} stating 
that $\cone$ has covering dimension equal to $1$ (note that this is not the case for 
$\cone[G]$ in general), and that Hawaiian earrings in spaces of dimension $1$ are homotopically 
embedded. This entails $\pi_1$-injectivity of the map $\Phi:\Earr \to \cone[G]$, hence of its lift 
$\Tilde{\Phi} : \Earr \to \cone[G]$; whence the theorem.

\subsection{Outline of the paper}
Section \ref{sec:preliminaries} recalls the algebraic and geometric background used in this work, 
including length estimates in $\SOL$. 
The groundwork being laid, we give in Section \ref{sec:setup} a precise statement of the main 
result, namely Theorem \ref{thm:technical_main}, which implies Theorem \ref{thm:main}. The 
rest of the paper is then devoted to the proof of Theorem \ref{thm:technical_main}, up to the end 
of Section \ref{sec:change_scale} which deals with Theorem \ref{thm:main_homology}. 

In Section \ref{sec:combinatorial}, we gather the results we need concerning the purely 
combinatorial properties of the words arising as iterated commutators. 
We then explain in Section \ref{sec:geom_loops} how to construct loops inside asymptotic cones by 
evaluating these words and taking ultralimits, and show that the first loop obtained is an 
embedding of the circle into $\cone$.

Next, we study in Section \ref{sec:higher_order_loops} loops coming from higher order commutators, 
which are no longer embedded in the asymptotic cone but yield spaces homotopically 
equivalent to finite wedge of circles.
Finally, Section \ref{sec:change_scale} deals with changing the scale of the loops in order to 
construct (subspaces homotopically equivalent to) Hawaiian earrings.
We obtain $\pi_1$-injectivity of the map to $\cone$, as well as $\Ho_1$-injectivity on a subgroup 
isomorphic to $\Z^\N$, completing the proof.

\subsection{Open questions}
\label{sec:open_questions}
We show that the SOL obstruction implies the existence of nonabelian free groups inside the 
fundamental group of the asymptotic cones; a natural question is then the following. Let $G$ be a 
Lie group without the SOL obstruction. Is it true that for every non-principal ultrafilter 
$\omega$, the fundamental group $\pi_1(\cone[G])$ is abelian?
And is this fundamental group nontrivial if  $G$ satisfies the homological obstruction (see Remark 
\ref{rq:obstruction_homologique})?
\smallskip

Moreover, the SOL obstruction can be restated, in the case of a real 
triangulable group $G$, as the property that the origin lies in the segment joining two weights of 
$H/[H,H]$, where $H$ is the exponential radical of $G$ 
(see Remark \ref{rq:SOL_obstruction_with_weights}). 
Similarly, define the ``$k$-SOL obstruction'' by asking that zero lies in the convex hull 
of $k+1$ weights of $H/[H,H]$, where $k\geq 1$; this is the negation of being $(k+1)$-tame as 
defined 
in \cite{cornulier2017geometric}.  
A natural generalization of the questions studied here is as follows: if $k$ is minimal such that 
$G$ has the $k$-SOL obstruction, is it true that the homotopy and homology groups of degree $k$ of 
$\cone[G]$ are ``very large'', say having continuum cardinality? 
Also, does the higher filling function of degree $k$ of $G$ have exponential growth?

\subsection{Acknowledgements}
This work is part of my PhD thesis under the supervision of Yves Cornulier and Bertrand Rémy.
I would like to express my deepest thanks to Yves Cornulier for introducing me to this delightful 
topic and for guiding me along the way.
I also thank Bertrand Rémy for his invaluable day-to-day support, and for many insightful 
discussions. Many thanks to Gabriel Pallier for introducing me to Figure \ref{fig:metric_view_SOL} 
and for discussions on the open questions.
I am grateful to the Laboratoire Jean Leray of Nantes Université for welcoming me during the start 
of this work. 

This work has been done entirely by myself, without the help of any large langage model or other 
artificial intelligence tools.

\section{Preliminaries} \label{sec:preliminaries}

\subsection{Asymptotic geometry of metric spaces} 
We describe here the background of our study. We study metric spaces under the lens of quasi-isometry, which we interpret as large-scale geometry.

Given a metric space $(X,d)$ and a positive scalar $\lambda$, we denote by $\lambda X$ the 
rescaled metric space $(X,\lambda d)$.

\subsubsection{Rough comparison}
Let us introduce here some notation for rough comparison, which will be of great use to state 
estimates of lengths and distances.

\begin{Def}[Rough comparison] \label{def:rough_comparison}
	Let $f, g :X \to \R_+$ be nonnegative functions on a set $X$. We write $f\preceq g$ and say 
	that $f$ is \emph{roughly bounded by $g$} if there exists a constant $C>0$ such that 
	\[
	f(x) \leq Cg(x)+C,~ \forall x \in X.
	\]
	We write $f \simeq g$ and say that $f$ and $g$ are \emph{roughly equal }whenever $f \preceq g$ 
	and $g \preceq f$.
\end{Def}

To simplify writing, we write (abusively) $f(x) \simeq g(x)$ to denote $f \simeq g$. For example, 
if $x,y \in \R_+$, then $\max(x,y) \simeq x+y$;  here the set $X$ is $\R_+ \times \R_+$.

\subsubsection{Quasi-isometries}
Let us now recall the following notion of large-scale equivalence between metric spaces.

\begin{Def} 
	Let $(X,d_X)$ and $(Y,d_Y)$ be metric spaces. A map $f:X\to Y$ is called a 
	\emph{quasi-isometry} if there exists a constant $C$ such that the following holds:
	\begin{enumerate}
		\item $\frac{1}{C} d_X(x_1,x_2) - C \leq d_Y(f(x_1),f(x_2)) \leq Cd_X(x_1,x_2) +C, ~ 
		\forall x_1,x_2 \in X ;$
		\item $\forall y \in Y, ~ \exists x \in X, ~ d_Y(y,f(x))\leq C.$ 
	\end{enumerate}
	If only Item $(i)$ is satisfied, meaning that $d_Y(f(x_1),f(x_2)) \simeq d_X(x_1,x_2)$, we say 
	that $f$ is a \emph{quasi-isometric embedding.} 
	If only the right-hand inequality holds, meaning that $d_Y(f(x_1),f(x_2)) \preceq 
	d_X(x_1,x_2)$, we say that $f$ is \emph{quasi-Lipschitz}. A quasi-isometric embedding of a 
	segment is called a \emph{quasi-geodesic segment}.
\end{Def}

The context of this paper is the study of \emph{quasi-isometry invariants of Lie groups}. There are 
many ways to endow a connected Lie group with a metric, using either a left-invariant Riemannian 
metric or a 
word metric coming from a compact generating subset. The  key point, which illustrates the 
flexibility of quasi-isometry, is that two choices of such metrics on a given Lie group yield 
quasi-isometric metric spaces by the Milnor-Schwarz lemma (see \cite[Theorem 
4.C.5]{cornulier2016metric}). Hence {Lie groups have a well-defined quasi-isometry type}.

\subsubsection{Asymptotic cones} 

Let us introduce here our main object of study, namely asymptotic cones. We are given a metric 
space $(X,d)$ whose large-scale geometry we want to study. To do this, we ``zoom out'' of $X$ by 
considering the sequence of rescaled metric spaces 
$(\frac{1}{n}X)_{n\in \N}=((X,\frac{d}{n}))_{n\in 
\N}$ and define the asymptotic cone of $X$ to be a ``limit'' of this sequence, in a sense described 
below. For more on asymptotic cones, see for example \cite{drutu2002quasi-isometry} or 
\cite[I.5]{bridson1999metric}. On the specific case of asymptotic cones of Lie groups, see 
\cite{cornulier2011asymptotic}.

Let us introduce the technical tool we use to take limits of metric spaces. By an 
\emph{ultrafilter} we mean a finitely additive probability measure $\omega$ on $\N$ taking values 
in $\{0,1\}$. It is \emph{non-principal} if it gives measure zero to all finite subsets of $\N$.
Their existence is ensured by Zorn's lemma, since non-principal ultrafilters are precisely maximal 
filters containing the Fréchet filter of co-finite sets (see \cite[I.5.48]{bridson1999metric}). 
Ultrafilters, working as ``universal extracting functions'', allow to make \emph{all} sequences in 
compact sets converge at the same time, in a consistent manner. 

\begin{Lem}[{see \cite[Lemma I.5.49]{bridson1999metric}}] 
	Let $(x_n)$ be a sequence taking values in a compact Hausdorff metric space $X$ and $\omega$ be 
	a non-principal ultrafilter. Then there exists a unique point $x_{\omega}$ in $X$, called the 
	\emph{ultralimit of $(x_n)$ along $\omega$} and denoted $\lim_{\omega}(x_n)$, such that
	\[
		\omega\left(\left\{n \in \N : d(x_n,x_{\omega})\leq \frac{1}{m}\right\}\right) =  1, 
		~\forall m\geq 1.
	\] 
	In particular, $x_{\omega}$ is an accumulation point of the sequence $(x_n)$. 
\end{Lem}

Let $(X,d)$ be a metric space and $\omega$ be a non-principal ultrafilter. We consider the set 
\[
\Precone(X) = \left\{ (x_n) \in X^{\N}:\exists C>0, d(x_0,x_n) \leq Cn, \forall n \in \N \right\} 
\]
and introduce the following equivalence relation on $\Precone(X)$:
\[
	(x_n)\sim_{\omega} (y_n) \iff \lim_{\omega}\frac{d(x_n,y_n)}{n}=0.
\]

\begin{Def}[Asymptotic cone] 
	The \emph{asymptotic cone of $X$ along $\omega$} is the metric space obtained as the quotient
	\[
		\cone[X] = \Precone(X)/\sim_{\omega}
	\]
	 endowed with the metric $d_{\omega}$ given by 
	 $
	 d_{\omega}([x_n],[y_n])=\lim_{\omega}\frac{d(x_n,y_n)}{n}
	 $
	 , where $[x_n]_n$ denotes the equivalence class of a sequence $(x_n)_n$.
\end{Def}

\begin{Rq}
	More generally, one could perform a similar construction using any sequence of pointed metric 
	spaces $(X_n,d_n,o_n)_{n\geq 0}$.
	The space obtained this way, endowed with the metric $d_{\omega}=\lim_{\omega}(d_n)$, is called the \emph{ultralimit of the sequence $((X_n,d_n,o_n))_n$ along $\omega$} and denoted $\lim_{\omega}(X_n,d_n,o_n)$. 
	In the case of asymptotic cones, it is implicit that the sequence of basepoints is constant, and the resulting metric space does not depend on the choice of the constant.
\end{Rq}

\begin{Ex}
	Let us give some well-known examples of asymptotic cones.
	\label{ex:asymptotic_cones}
	\begin{enumerate}
		\item If $X$ is $\R^k$ endowed with a norm, then the spaces $(X,\frac{1}{n}d)$ 
		are all isometric, and all of its asymptotic cones are isometric to $\R^k$ with the same 
		norm.
		\item If $X$ has strictly negative curvature, say $X$ is a tree or the hyperbolic space 
		$\H^k$, then all of its cones have infinitely negative curvature. This means that they are 
		(geodesic) metric spaces where every triangle is a tripod; these are called \emph{real 
		trees}. In fact, it is shown in \cite{dyubina2001explicit} that every complete simply 
		connected manifold of dimension $\geq 2$ and curvature bounded above by a negative constant 
		has all of its asymptotic cones isometric to the universal real tree of degree 
		$2^{\aleph_0}$.
		\item If $X$ is $\Z^k$ endowed with the word metric associated to its standard generating 
		subset, then all of the asymptotic cones of $X$ are isometric to $\R^k$ with the $l^1$ 
		metric. This is a special case of the fact that asymptotic cones of discrete groups of 
		polynomial growth are Carnot-nilpotent Lie groups endowed with sub-Finsler metrics, 
		as established by Pansu in \cite{pansu1983croissance}. For more on this, see 
		\cite{ledonne2025metric}.
	\end{enumerate} 
\end{Ex}

\begin{Rq}
	Asymptotic cones were introduced by Gromov in \cite{gromov1981groups} for the proof of his theorem on groups of polynomial growth, but using Gromov-Hausdorff convergence. The use of ultrafilters, mandatory to get any kind of convergence in the case of exponential growth, came later in \cite{vandendries1984gromovs}.
\end{Rq} 

Maps between metric spaces naturally induce maps between asymptotic cones, in the following way.

\begin{Prop}[{see \cite[Lemma 10.48]{drutu2018geometric}}] 
	\label{prop:maps_induced_between_cones}
	Let $f :(X,d_X) \to (Y,d_Y)$ be a quasi-Lipschitz map between metric spaces. Let $\omega$ be a non-principal ultrafilter and $X_{\omega}$ (resp. $Y_{\omega}$) be the asymptotic cone of $X$ (resp. $Y$) along $\omega$.
	
	Then $f$ induces a continuous map $f_{\omega}$ between asymptotic cones, defined by
	\[
		f_{\omega} : X_{\omega} \to Y_{\omega},~ [x_n] \mapsto [f(x_n)].
	\]
	With this construction, quasi-isometries (resp. quasi-isometric embeddings, resp. quasi-Lipschitz maps) induce bi-Lipschitz homeomorphisms (resp. bi-Lipschitz embeddings, resp. Lipschitz maps) between asymptotic cones. 
\end{Prop}

\begin{Rq}
	In practice, we will often consider sequences of maps $(f_n :X_n \to Y)_{n \geq 0}$ and 
	consider the 
	induced map between rescaled ultralimits $f_{\omega} : \lim_{\omega}(X_n,d_n/n) \to \cone[Y]$ 
	defined by $f_{\omega}([x_n]) = [f_n(x_n)]$. The map $f_\omega$ is well-defined and 
	Lipschitz-continuous if the sequence 
	$(f_n)$ is \emph{uniformly} quasi-Lipschitz, 
	and is a bi-Lipschitz embedding whenever $f_n$ is a 
	quasi-isometric embedding for all $n\geq 0$ with constants uniform in $n$.
\end{Rq}

A direct corollary of this proposition, and the main feature of interest of asymptotic cones, is 
that \emph{the bi-Lipschitz class of the asymptotic cone is a quasi-isometry invariant of metric 
spaces}. In particular, the topological invariants of the cone are large-scale invariants of the 
initial metric space. 

\begin{Rq}
	In the sequel, we fix once and for all a non-principal ultrafilter $\omega$.
	 While it is known that the asymptotic cone can depend on the choice of an ultrafilter, we will 
	 almost only deal with statements holding for \emph{every} non-principal ultrafilter.
	
	Note moreover that, in the case of connected Lie groups, if the continuum hypothesis holds then 
	the bi-Lipschitz class of the asymptotic cone does not depend on the non-principal ultrafilter 
	(see \cite[Corollaire 1.9]{cornulier2014aspects}).
\end{Rq}

\subsection{Groups with the SOL obstruction } 

Recall that we study here real triangulable groups, which form a class of groups that contains a representative for every quasi-isometry class of connected Lie groups. In particular, they are simply connected solvable Lie groups.

\subsubsection{The groups $\SOL$} \label{subsec:SOL_alpha}
As stated before, our base groups all come from a family of deformations of the classical group $\SOLV$.

\begin{Def}
	Given $\alpha>0$, we denote by $\SOL$ the Lie group given by the semidirect product $\R^2 \rtimes_{(1,-\alpha)} \R$, where $\R$ acts on $\R^2$ via the one-parameter subgroup
	\[ t \mapsto  \begin{pmatrix}
							e^t & 0 \\
							0 & e^{-\alpha t}
					   \end{pmatrix}.
	\]
\end{Def}

The key assumption here is that the so-called \emph{weights} $1$ and $-\alpha$ have opposite signs, so that $t>0$ expands the first coordinate $x$ and contracts the second coordinate $y$ at the same time.

Notice that $\SOL$ is homeomorphic to $\R^3$. Thus, the algebraic difference does not appear at the 
topological level; but it does at the geometric level. In particular, the left-invariant Riemannian 
metric that is given by the identity matrix at the point $(0,0,0)$ is given by the formula
\[
 	ds^2=dt^2 + e^{-2t}dx^2 + e^{2\alpha t}dy^2.
\] 
Notice how the expanding action of $t$ on $x$, for $t>0$, is reflected geometrically by the 
contracting factor $e^{-2t}$ in front of $dx^2$. 
See Figure \ref{fig:metric_view_SOL} for a metric view of $\SOL$.

\begin{Rq}
	Notice that the $xt$ plane in $\SOL$, $\sg[x]\rtimes \sg[t]$, is a {totally geodesic copy of 
	the hyperbolic plane $\H^2$}. 
	Indeed, it is the Lie subgroup $\R \rtimes\R$, where $t$ acts by $e^t$ and with induced metric 
	$dt^2+e^{-2t}dx^2$, which is the $\log$-model for the hyperbolic plane: an explicit isometry 
	with the upper half-plane model of $\H^2$ is merely given by $(x,t) \mapsto (x,e^t)$.
	Moreover, it is the fixed point set of the isometry $(x,y,t) \mapsto (x,-y,t)$,
	hence this submanifold is totally geodesic.
	
	Similarly, the $yt$ plane is a totally geodesic submanifold in $\SOL$ which is bi-Lipschitz 
	homeomorphic to $\H^2$. 
	The main difference is that it is ``flipped upside-down'', meaning one has to change the sign 
	of $t$ (and rescale it by $\alpha$) to get back to the $\R \rtimes\R$ model of $\H^2$.
	For more on the geometry of $\SOL$, see \cite[Section 3]{eskin2012coarseI}.
\end{Rq}
The key thing to note is that we have a double foliation of $\SOL$ by hyperbolic planes, one 
contracting ``upwards'' and the other contracting ``downwards''. 
This observation is crucial to our metric computations in Section \ref{subsec:image_SOL}, helping 
us construct many quasi-geodesics inside $\SOL$.

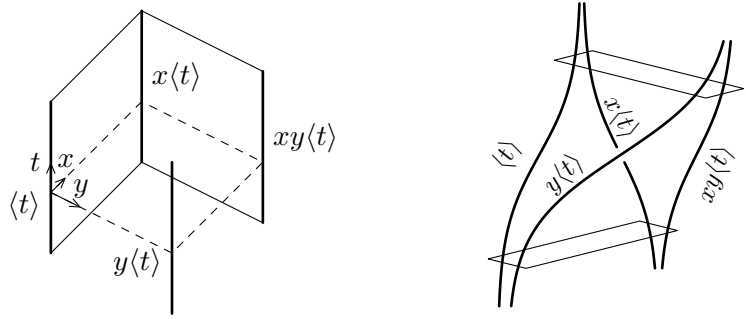
\begin{figure}
	\centering
	\raisebox{20pt}{
	 \begin{tikzpicture}[line cap=round,line join=round,>=angle 45,x=0.4cm,y=0.4cm, 
	 scale=.5, ]
		%\draw[help lines, color=gray, dashed] (-15,-15) grid (15,10);
		\clip(-9,-15) rectangle (14,7); %découpe au délà d'une certaine zone
		% Verticales
		\draw[line width=1pt] (-6,0)-- (-6,-10); 
		\draw[line width=1pt] (8,2)-- (8,-8);
		\draw[line width=1pt] (0,6)-- (0,-4); 
		\draw[line width=1pt] (2,-14)-- (2,-4);
		%Horizontales
		\draw (-6,0)-- (0,6); %Pente (+6,-6) = 1
		\draw (0,6)-- (8,2); %Pente (+8,-4) = -1/2
		\draw (-6,-10)-- (0,-4); % (+6,+6) =1
		\draw (8,-8)-- (0,-4); % (-8,+4) = -1/2
		%\draw[dashed] (-6,-10) -- (2,-14); % (+8,-4) = -1/2
		%\draw[dashed] (2,-14) -- (8,-8) ; % (+6,+6) =1
		%\draw[dashed] (-6,0) -- (2, -4); % (+8,-4) =-1/2
		%Labels
		\draw (-6,-8.5) node[anchor=south east] {$ \sg[t] $};
		\draw (0,0) node[anchor=south west] {$ x \sg[t] $};
		\draw (2,-9) node[anchor=north east] {$ y \sg[t] $};
		\draw (8,-4) node[anchor=south west] {$ xy\sg[t] $};
		\draw [dash pattern=on 3pt off 3pt] (0,0)-- (-6,-6)-- (2,-10)-- (8,-4) -- cycle;
		%Vecteurs
		\draw [->] (-6,-6) -- (-5,-5) node[above]{$x$};
		\draw [->] (-6,-6) -- (-4,-7) node[above]{$y$};
		\draw [->] (-6,-6) -- (-6,-4) node[left]{$t$};
	\end{tikzpicture}
	}
	\qquad
	\begin{tikzpicture}[line cap=round,line join=round,>=angle 45,x=0.5cm,y=0.5cm, ]
		\clip(-2,-6) rectangle (8,6);
		%Côtés
		\draw [color=black, line width=1pt,domain=-1.5:1.5, samples = 200, variable = \t] 
		plot({5-tanh(\t)},{-2*\t});
		\draw [color=black, line width=1pt,domain=-1.5:2, samples = 200, variable = \t] 
		plot({5-(3*tanh(\t)+2)},{-2*\t});
		\draw [color=black, line width=1pt,domain=-2:-0.1, samples = 200, variable = \t] 
		plot({5-(2-tanh(\t))},{-2*\t});
		\draw [color=black, line width=1pt,domain=0.1:1.5, samples = 200, variable = \t] 
		plot({5-(2-tanh(\t))},{-2*\t});
		\draw [color=black, line width=1pt,domain=-2:2, samples = 200, variable = \t] 
		plot({5-(1.1*tanh(\t)+4.15)},{-2*\t});
		%Rectangles
		\draw [shift={(1.25,2.5)}] (0,0) -- (1,0.25) -- (5,-0.75) -- (4,-1) -- cycle;
		\draw [shift={(-0.5,-2.75)}] (0,0) -- (1,-0.25) -- (5,0.75) -- (4,1) -- cycle;
		%Labels
		\draw (0,0) node[anchor=center, rotate=65] {\small $\sg[t]$};
		\draw (1.5,-0.5) node[anchor=center, rotate=45] {\small $y\sg[t]$};
		%$\{(0,1)\}\times \R$
		\draw (5.5,-0.5) node[anchor=center, rotate=70] {\small $xy\sg[t]$};
		%$\{(1,1)\}\times \R$
		\draw (3,1) node[anchor=center, rotate=-60] {\small $x\sg[t]$};
		%$\{(1,0)\}\times \R$
	\end{tikzpicture}
	\vspace*{-10pt}
	\caption{
	On the right, metric view of the tube $[0,1]^2 \times \R$ inside $\SOL$, viewed on the left in 
	coordinates. 
	Each of the four faces is a portion of $\H^2$, oriented alternatively up or 
	down, and the edges are left translates of $\sg[t]$ that are geodesics in $\SOL$ and in $\H^2$. 
	Two edges of a face either converge or 
	diverge exponentially when $t$ goes to $\pm \infty$.
	The rectangles are the metric view of squares $[0,1]^2\times \{T\}$, with $T$ either 
	positive or negative.
	Figure adapted from \cite[Fig. 11]{ledonne2024rough}.}
	\label{fig:metric_view_SOL}
\end{figure}

\subsubsection{The SOL obstruction}
The following condition was introduced by Abels in \cite{abels1987finite} in a $p$-adic context and 
in terms of weights, and then made explicit by Cornulier-Tessera in \cite{cornulier2017geometric}.

\begin{Def}
	A real triangulable Lie group is said to have the \emph{SOL obstruction} if it admits a group $\SOL$ as a quotient for some $\alpha>0$.
\end{Def}
Loosely speaking, the SOL obstruction is the property of having both dilating and contracting 
directions, just like what happens inside $\SOL$; see Figure \ref{fig:metric_view_SOL}.
The negation of this condition is sometimes 
called being \emph{$2$-tame}.

\begin{Rq} \label{rq:SOL_obstruction_with_weights}
	The SOL obstruction was originally defined in \cite{abels1987finite} using the weights of the 
	Lie algebra. Those are 
	linear functionals on a Cartan subalgebra encoding the dynamical properties of the conjugation 
	action. 
	
	More precisely, in the case of real triangulable groups the SOL obstruction is equivalent to 
	having $0$ lie on the segment between two weights of the action on the 
	abelianization of the exponential radical; 
	see \cite[Proposition 4.D.9]{cornulier2017geometric}.
	In particular, this holds for any Lie group of the form $N\rtimes A$, with $N$ nilpotent and 
	$A$ abelian, where the action of $A$ on $N/[N,N]$ has two nonzero weights $\lambda,\mu$ such 
	that $0$ 
	lies on the segment $[\lambda,\mu]$. 
\end{Rq} 

\begin{Ex} \label{ex:groups_SOL_obstruction}
	Aside from the groups $\SOL$, groups with the SOL obstruction include for example the so-called 
	\emph{oscillator} groups
	\[
		\Osc=\Heis_3\rtimes_{(1,-\alpha,1-\alpha)}\R, ~\alpha>0,
	\]
	where $\Heis_3$ is the $3$-dimensional Heisenberg group and $t\in \R$ acts via the matrix 
	$\Diag(e^t,e^{-\alpha t},e^{(1-\alpha)t})$. Modding out $\Osc$ by the center of $\Heis_3$ 
	indeed yields $\SOL$.
	
	Another class of examples consists of the semi-direct products $H \rtimes \SOL$, where $H$ is 
	any real triangulable group.  
\end{Ex}

Note that, in the case of semi-direct products, Theorem \ref{thm:main} follows from 
\cite{burillo1999dimension}, since $\cone[H \rtimes \SOL]$ retracts onto $\cone$. However, this is 
not the case for non-split extensions of $\SOL$ such as the oscilators $\Osc$, for which our 
results are new.

As exposed in the introduction, our interest in the SOL obstruction comes from the following 
result. 

\begin{Thm}[{\cite[Thm 12.C.1]{cornulier2017geometric}}] \label{thm:SOL_obs_dehn_exp}
	Let $G$ be a real triangulable group with the SOL obstruction. Then the Dehn function of $G$ is at least exponential.
	
	In particular, by Theorem \ref{thm:cone_dehn}, there exists an asymptotic cone of $G$ whose fundamental group is non-trivial.
\end{Thm}

This theorem motivates the further study of fundamental groups of asymptotic cones of groups with 
the SOL obstruction. This is what we deal with here, as Theorem \ref{thm:main} 
gives a more precise result about these fundamental groups, showing that they are far from 
being trivial as they contain nonabelian free groups.

\begin{Rq}[Asymptotic cone of the groups $\SOL$]
	Although this will not be of use here, note that the asymptotic cone $\cone$ can be explictly 
	described, as a horocyclic product. 
	It is a ``Diestel-Leader $\R$-graph'', meaning the ``hypersurface'' of 
	equation $b(x) + \alpha b(y)=0$ inside the product of two homogeneous real trees, where $b$ is 
	a Busemann function. In particular this does not depend on $\omega$, and 
	up to bi-Lipschitz equivalence this does not depend on $\alpha$ either.
	
	This description comes from the quasi-isometric embedding of  $\SOL$ 
	inside the product of two hyperbolic planes as a subspace with a similar equation. 
	The real tree appears as the asymptotic cone of $\H^2$, see Example \ref{ex:asymptotic_cones}. 
	For more on this, see \cite[Section 9]{cornulier_dimension_2008} ; for the link 
	with lamplighters and Diestel-Leader graphs, see \cite[Section 3]{eskin2012coarseI}.
\end{Rq}

\subsubsection{The exponential radical} \label{subsec:exp_rad}
To make use of the solvability of our groups, we consider a normal nilpotent subgroup defined as 
follows.

\begin{Def}
	Let $G$ be a simply connected solvable Lie group. The \emph{exponential radical} of $G$, denoted $\Rexp$, is the subgroup of elements that are \emph{exponentially distorted}, meaning
	\[
		\Rexp \defeq \left\{ g \in G :  |g^n|  \preceq \log(1+n) \right\}.
	\]
\end{Def}

Defined by Guivarc'h in \cite{guivarch1980loi} and studied by Osin in \cite{osin2002exponential}, it has in our context the following characterization.

\begin{Lem}[{\cite[Lemma 2.5]{cornulier_dimension_2008}}]   
	Let $G$ be a real triangulable group. Then the exponential radical $\Rexp$ is a nilpotent normal subgroup of $G$ equal to the stable term of the descending central series of $G$, meaning that
	\[
		\Rexp = \bigcap_{i\geq 1}\mathrm{C}^iG.
	\]
	Equivalently, $\Rexp$ is the smallest normal subgroup of $G$ such that $G/\Rexp$ is nilpotent. 
\end{Lem}

\begin{Ex}
	Inside $\SOL$, the exponential radical is the $xy$ plane. In the oscillator group 
	$\mathrm{Osc}_{\alpha}$ (see Example \ref{ex:groups_SOL_obstruction}), it is the subgroup 
	$\mathrm{Heis}_3$.
\end{Ex} 

\begin{Lem} \label{lem:image_exp_rad}
	Let $p:G\to H$ be a surjective homomorphism between real triangulable groups. Then we have 
	\[
		p(\Rexp) =\Rexp[H].
	\] 
\end{Lem}

\begin{proof}
	Since $p$ preserves commutators, we get by induction that $p(\mathrm{C}^iG)=\mathrm{C}^iH$ for 
	all $i\geq 0$. Taking the direct image yields the inclusion $p(\Rexp) \subset \Rexp[H]$.
	
	For the reverse inclusion, see that $p(\Rexp)$ is normal in $H$ since $\Rexp$ is normal 
	in $G$ 
	and $p$ is surjective. Moreover, notice that $H/p(\Rexp)$ is a quotient of $G/\Rexp$. The 
	latter being nilpotent, $H/p(\Rexp)$ is nilpotent. This implies that $\Rexp[H] \subset 
	p(\Rexp)$, completing the proof.
\end{proof}

\begin{Rq}
	This lemma is what allows us to work with the exponential radical in Section 
	\ref{sec:geom_loops}. What it means in practice is that, given a real triangulable group $G$ 
	quotienting onto $\SOL$, the elements $x,y$ of $\SOL$ can be lifted to elements 
	$\Tilde{x},\Tilde{y}$, not only in $G$, but \emph{in the exponential radical of $G$}. Hence, by 
	nilpotency of this subgroup, iterated commutators of $\Tilde{x}$ with $\Tilde{y}$ eventually 
	vanish; see Lemma \ref{lem:theta_k_is_a_loop_in_G}.
\end{Rq}

\subsection{Combinatorial paths and loops} 
\label{subsec:combinatorial_paths}

The bulk of this paper consists in constructing the loops we need to get non-trivial fundamental groups and checking their metric properties. To achieve this, we take a combinatorial point of view. Instead of directly defining loops as maps from the circle, we construct words in a free group that we then evaluate into our groups. This idea appears already in \cite[Section 12]{cornulier2017geometric}, but we make it more explicit for our needs.

While the following formalism might seem cumbersome at first, it will allow us to highlight in our 
proof the distinction between combinatorial arguments (in Sections \ref{sec:combinatorial}) and 
geometric ones (in Sections \ref{sec:geom_loops} to \ref{sec:change_scale}).

\subsubsection{Combinatorial paths in groups}

Let $G$ be a locally compact group, generated by a symmetric compact subset $S$. We endow $G$ with the word metric relative to the generating subset $S$.

By a \emph{combinatorial path} of length $\ell$ in $G$ we mean a finite sequence of points 
$\theta=(\theta(i))_{0\leq i \leq \ell}$ in $G$ such that for all $i\leq \ell-1$, the elements 
$\theta(i)$ and $\theta(i+1)$ are adjacent in the Cayley graph of $G$ with respect to $S$; this 
simply means that there exists $s_i \in S$ such that $\theta(i+1)=\theta(i)s_i$. 

\begin{Not}
	Let $\theta$ be a combinatorial path. We denote by $\overline{\theta}$ the reverse path.
	If $g \in G$, $g\theta$ denotes the left \emph{translate of $\theta$ by $g$}, meaning the path 
	defined by left translation of all points of $\theta$. 
	
	Two paths $\theta$ and $\theta'$ of  respective lengths $\ell$ and $\ell'$ are 
	\emph{concatenable} if $\theta(\ell)$ is equal to  $\theta'(0)$. In that case, we denote by 
	$\theta \theta'$ their concatenation, which is a path of length $\ell+\ell'$.
	
\end{Not}

\begin{Rq}[Paths and loops as maps on segments and circles] \label{rq:paths_as_maps_on_segments}
	From our definition, a combinatorial path $\theta$ of length $\ell$ in a group $G$ is a map 
	$\theta:\{0,\dots, \ell\} \to G$. We view the discrete metric space $\{0,\dots, \ell\}$ as the 
	integer points of the segment $[0,\ell]$, and abusively write $\theta$ as a map 
	$\theta:[0,\ell] \to G$. 
	
	This is justified by the fact that we wish to deal with asymptotic geometry of paths. In particular, if $\ell_n$ is a sequence of positive integers satisfying $\ell_n \simeq n$, then the sequence of spaces $\frac{1}{n}  \{0,\dots, \ell_n\}$ has the same limit (in the Gromov-Hausdorff metric, and hence the same ultralimit) as the sequence $\frac{1}{n} [0,\ell_n]$, namely the unit segment $[0,1]$.
	
	Similarly, a combinatorial loop of length $\ell$ in $G$ is a map $\Z/\ell\Z \to G$. We view 
	this discrete metric space as the integer points of the circle of length $\ell$, which we 
	denote by $\ell \S^1 = \R/\ell \Z$. Again we allow ourselves to write $\theta$ as a map $\ell 
	\S^1 \to G$, as the maps induced between asymptotic cones are unchanged.
\end{Rq}

From our definition of combinatorial paths, we see that consecutive points on a path must be at distance $1$ with respect to the word metric relative to the generating subset $S$. Hence the following.

\begin{Lem} \label{lem:paths_are_lipschitz}
	Let $\theta$ be a combinatorial path of length $\ell$ in a locally compact group $G$. Then the 
	map $\theta :[0,\ell] \to G$ is $1$-Lipschitz. 
\end{Lem}

In particular, a sequence $(\theta_n)_{n\geq 0}$ of paths of lengths $\ell_n \simeq 
n$ induces 
a $1$-Lipschitz map $\theta = \lim_{\omega}\theta_n : [0,1] \to \cone[G]$, where the segment 
$[0,1]$ is (up to a bi-Lipschitz homeomorphism) the ultralimit of the rescaled sequence
$(\frac{1}{n}[0,\ell_n])_{n\geq 0}$.

\subsubsection{Evaluating words on generators}
Consider now $F_S$, the free group over $S$, so that there is a canonical surjective homomorphism $\ev_S : F_S \to G$.

\begin{Def} \label{def:evaluation_mots_générateurs}
	Let $w=s_1s_2\dots s_{\ell}$ be an element of $F_S$. We call the \emph{evaluation of $w$ in $G$}, denoted $\lev w \rev$, the combinatorial path $(\lev w \rev(i))_{0\leq i \leq \ell}$ based at $1$ in $G$ defined by the sequence
	\begin{align*}
		1, ~ \ev_S(s_1), ~ \ev_S(s_1s_2), ~ \dots ~ \ev_S(s_1s_2\dots s_{\ell}),
	\end{align*}
	meaning that $\lev w \rev = \ev_S(s_1s_2\cdots s_i)$ for $0\leq i \leq \ell$. 
\end{Def}

We will often view the $s_i$ as elements of $G$ and omit writing the projection from $F_S$ to $G$. We also drop the subscript $G$ whenever there is no ambiguity on the group of evaluation.

\begin{Rq}
	The endpoint of $\lev w \rev[]$ is equal to the projection of $w$ in $G$.
	In particular, the path $\lev w \rev[]$ is a \emph{combinatorial loop} whenever $s_1\dots 
	s_{\ell} =1$ in $G$, i.e. when $w$ is a relation in $G$. 
	In that case, the evaluation of the word $w^{-1}$ in $G$ is the reverse loop $\overline{\lev 
	w\rev[]}$. 
\end{Rq}

To get paths based at different points, we will often consider, for $g\in G$, the left translate 
$g\lev w \rev[]$, which is a path based at $g$.

As we wish to evaluate products of words in $F_S$, we introduce the following terminology.

\begin{Def} \label{def:cancellation_length}
	Let $w,w'$ be two words in $F_S$, and let $| \cdot |$ denote the word length. We call 
	\emph{cancellation length} of $w$ with $w'$ the non-negative quantity $|w|+|w'|-|ww'|$. 
\end{Def}

Here, $ww'$ is the product of $w$ and $w'$ in $F_S$, meaning the reduced concatenation of these two words. In particular, the cancellation length is always an even integer, and is zero if and only if the concatenation of $w$ with $w'$ is already a reduced word.

Now the following proposition and its corollary allows us to study metric properties of evaluations of products of words. 

\begin{Prop}[Evaluation of a product] \label{prop:evaluation_product}
	Let $G$ be a locally compact group with a compact generating set $S$. Let $w$ and $w'$ be two words in the free group $F_S$, and let $g$ and $g'$ be their respective projections in $G$. 
	Denote by $\delta$ their cancellation length. Consider the following two paths from $1$ to $g'$ in $G$ :
	\begin{itemize}
		\item $p=(\lev w \rev[]) (g\lev w' \rev[]) $, the concatenation of the evaluations of $w$ and $w'$ in $G$;
		\item $q=\lev ww' \rev[]$, the evaluation in $G$ of the product of $w$ and $w'$ in $F_S$.
	\end{itemize}
	Then $p$ and $q$ 
	differ only by a path of length $\delta/2$ run back and forth.
\end{Prop}

\begin{proof}
	Write $w=w_1z$ and $w'=z^{-1}w'$ with $z$ of maximal length. Then $ww'=w_1w_2$ and the concatenation of $w_1$ with $w_2$ is reduced, hence $\delta=2|z|$. 
	
	Let $g_1$ be the projection of $w_1$ in $G$. 
	Denote $c_1=\lev w_1 \rev[]$, $c_2=g_1 \lev w_2\rev[]$ and $s=g_1\lev z \rev[]$.
	Then the paths $p$ and $q$ satisfy $q=c_1c_2$ and $p=c_1s\overline{s}c_2$.
	In particular, they differ only be the concatenation $s\overline{s}$, where $s$ is a path of length $|z|=\delta/2$.
\end{proof}

\begin{Cor}[Filiform vanishing] \label{cor:filiform_vanishing}
	Let $G$ be a locally compact group with a compact generating set $S$.
	Let $w_n, w_n'$ be two sequences of words in $F_S$ of lengths $\simeq n$ such that the cancellation lengths $\delta_n$ of $w_n$ with $w_n'$ satisfy $\delta_n/n \xrightarrow{n \to \infty} 0$. Let $g_n$ denote the projection of $w_n$ in $G$.
	
	Then the corresponding paths $p_n=\lev w_n \rev[] (g_n\lev w_n' \rev[]) $ and $q_n=\lev w_n w_n' \rev[]$ induce the same continuous map $[0,1] \to \cone[G]$ between asymptotic cones.
\end{Cor}

We call \emph{filiform vanishing} this phenomenon that relates concatenation of evaluations to 
evaluation of concatenations, since it consists of segments run back and forth which vanish
asymptotically.

\begin{proof}
	The combinatorial paths $p_n$ and $q_n$ have respective lengths $|w_n|+|w_n'|$ and $|w_n|+|w_n'|-\delta_n$, which are both $\simeq n$ by our assumptions. Viewing them as maps from integer segments to $G$, both are $1$-Lipschitz. 
	Hence they induce continuous maps from $[0,1]=\lim_{\omega} \left(\frac{1}{n}[0,n]\right)$ to $\cone[G]$, say $p_{\omega}$ and $q_{\omega}$. 
	
	Now by Proposition \ref{prop:evaluation_product}, $p_n$ and $q_n$ differ only by a segment 
	$s_n$ of length $\delta_n/2$. Since $\delta_n/n$ tends to zero, the ultralimit of the sequence 
	$(s_n)$ is a single point. 
	Therefore $p_{\omega}=q_{\omega}$. 
\end{proof}

\subsubsection{Evaluating words on elements} 
We wish to evaluate words not only on generators, but on arbitrary elements of the group.  For $n\geq 1$, consider the free group $F_n$ over $n$ generators $C_1,\dots C_n$. For any $n$-uple $\alpha =(\alpha_1,\dots \alpha_n) $ of elements of $F_S$, we have a natural homomorphism $\ev_{\alpha} : F_n\to F_S$ that sends each generator $C_i$ to the word $\alpha_i$.

\begin{Def} \label{def:evaluation_mots_elements}
	Let $\alpha_1,\dots \alpha_n$ be elements of $F_S$ and
	let $W$ be an element of $F_n$. 
	By the \emph{evaluation of $W$ on  $\alpha=(\alpha_1,\dots ,\alpha_n)$ in $G$}, 
	denoted $\lev W(C_1=\alpha_1,\dots C_n=\alpha_n) \rev$, 
	we mean the combinatorial path $\lev \ev_{\alpha}(W) \rev$ in $G$ obtained by first evaluating $C_i$ on $\alpha_i$ and then evaluating the word $\ev_{\alpha}(W)$ in $G$. 
\end{Def}

\begin{Rq}
	Note that the length $\ell$ of the path $\lev W((C_1=\alpha_1,\dots C_n=\alpha_n)) \rev$ satisfies
	\[
		\ell \leq |W|. \max_i |\alpha_i|_S,
	\]
	where $|\cdot|$ (resp. $|.|_S$) denotes the word length in $F_n$ (resp. $F_S$).
\end{Rq}
For convenience, we will often identify the words $\alpha_i$ with their projections $g_i$ in $G$, 
and abusively write $\lev W(C_1=g_1,\dots C_n=g_n) \rev$.

\begin{Ex}[Evaluating words] 
	Let $H=\R\rtimes\R$ be the Lie group model of the hyperbolic plane $\H^2$. Let $x=(1,0)$ and $t=(0,1)$ be basis elements of $H$. Say we consider the group word $V$ over generators $U,X$ defined by 
	\[
		V = UXU^{-1}.
	\]
	Its evaluation on $U=t^n,X=x$ in $H$, namely $\lev V(U=t^n,X=x) \rev[H] $, is the path 
	$v_n=\lev t^nxt^{-n}\rev[H]$ of length $\ell=2n+1$.
	Explicitly, it is the following sequence of points:
	\[
		v_n = (1,t,t^2, \dots t^n, t^nx, t^nxt^{-1}, \dots, t^nxt^{-n}),
	\]
	where the group words over $t$ and $x$ are viewed here as elements of $H$. 
	Formally, we may view $v_n$ as a map $\{0, \dots 2n+1\} \to \H^2$. In fact, we will see in Section \ref{sec:geom_loops} that this path is a quasi-geodesic segment, with quasi-isometry constants independent of $n$.
\end{Ex}	

\begin{Not} \label{not:capitals_for_words_lowercase_for_evaluations}
	Our convention, whenever possible, is to use capital letters for words in the free group and 
	the corresponding lowercase for their evaluations, adding a subscript whenever there is a 
	dependence in the evaluation. For example, see the preceding example with the word $V$ and its 
	evaluations $v_n$, or $\Theta_k$ and the paths $\theta_{k,n}$ in Section \ref{sec:geom_loops}.
\end{Not}

\subsection{Estimates of length in the groups $\SOL$}\label{sec:length}

We deal here with estimates of length in solvable Lie groups. To turn them into metric spaces, we view them as compactly generated groups and use word-length metrics, just like with discrete groups. Any two left-invariant quasi-geodesic metrics on such a group are quasi-isometric, so the following estimates do not depend on our choices. In particular, these statements also hold when considering distances coming from left-invariant Riemannian metrics.

Recall that in a pointed metric space $(X,d,o)$ , we denote by $|x| \defeq d(x,o)$ the 
\emph{length} of $x$. For example, in a group $G$ endowed with a left-invariant metric, $|g|\defeq 
d(g,1_G)$.

The following length estimate is a key ingredient in the metric computations of Sections 
\ref{sec:geom_loops} to \ref{sec:change_scale}.

\begin{Prop}[{\cite[Thm. 6.B.2]{cornulier2017geometric}}] \label{prop:dist_SOL} 
	In $\SOL=\R^2 \rtimes_{(1,-\alpha)} \R$, we have the following estimate:
	\[
	|(x,y,t)| \simeq |t| + \log\left(1+|x|+|y|\right), ~ \forall (x,y,t) \in \SOL.
	\]
\end{Prop}

\begin{Rq}
	Note that for nonnegative numbers $u$ and $v$, we have the rough equality 
	\[
	\log(1+u+v) \simeq \log(1+u) + \log(1+v), 
	\]
	allowing us to consider the $x$ and $y$ coordinates separatly in the estimate from the 
	proposition. The former estimate can be obtained by taking logarithms in the following 
	inequality, holding for $u,v\geq 0$: $ 1+u+v \leq (1+u)(1+v) \leq (1+u+v)^2$.
\end{Rq}

\begin{Rq}
	This proposition is a restatement of \cite[Thm. 6.B.2]{cornulier2017geometric} in the case of 
	the group $\SOL$. The general theorem gives estimates of length in solvable Lie groups of the 
	form $G=U\rtimes A$ with $U$ unipotent and $A$ abelian, using contracting actions in so-called 
	\emph{tame subgroups} which have nonpositive curvature properties.
	
	Here, $U=\R^2$ is the $xy$-plane and $A=\R$ is the $t$-line. The tame subgroups are 
	$U_1=\sg[x]\rtimes \sg[t]$ and $U_2=\sg[y]\rtimes \sg[t]$ which are both quasi-isometric to 
	$\H^2$, so indeed of negative curvature.
\end{Rq}

We give the proof of Proposition \ref{prop:dist_SOL} to highlight the key ideas in this simple case.

\begin{proof}
We use the word metric associated to the compact generating subset $S=[-1,1]^3$ of $G=\SOL$. We 
then establish the two rough inequalities between $|(x,y,t)|_S$ and $|t|+\log(1+|x|+|y|)$.

\emph{Inequality $\succeq$.} We take $g=(x,y,t) \in G$ of length $n\geq 1$, and want to show that its coordinates satisfy $\log(1+|x|+|y|) \preceq n$ and $|t|\preceq n$. Since $g$ has length $n$, it can be written as $g=s_1\dots s_n$ with $s_i \in S$. We write $s_i=(x_i,y_i,t_i)$ with $x_i,y_i,t_i$ in $[-1,1]$. 
Using the group law of $\SOL$, we then compute
\[
	s_1s_2=(x_1+e^{t_1}x_2, y_1+e^{-\alpha t_1}y_2,t_1+t_2).
\]
Repeating this computation, we get by induction the equalities
\begin{align*}
	&t = t_1 + \dots t_n;\\
	&x = x_1 + e^{t_1}x_2 + e^{t_1+t_2}x_3 + \dots + e^{t_1+\dots +t_{n-1}}x_n; \\
	&y = y_1 + e^{-\alpha t_1}y_2 + e^{-\alpha(t_1+t_2)}y_3 + \dots + e^{-\alpha(t_1+\dots +t_{n-1})}y_n. 
\end{align*} 

Hence we have $|t|\leq n$,  and $|x|$ and $|y|$ are bounded by $ne^{Cn}$, where 
$C=\max(\alpha,1)$. Taking the logarithm 
of the sum, we get that $\log(1+|x|+|y|)\preceq Cn+\log(n) \preceq n$. 

\emph{Inequality $\preceq$.} Let now $g=(x,y,t) \in G$. 
The goal is to write $g$ as a product of $n$ elements of $S$, where $n\preceq |t| +\log(1+|x|+|y|)$. It is enough to do it separatly for each coordinate. 
First, $(0,0,t).(0,0,1)^{-\lfloor t \rfloor} = (0,0,t-\lfloor t \rfloor)$ lies in $S$, so we have 
$|(0,0,t)|_S \preceq |t|$. 

Next, we use the contracting action of $t$ on $x$ to deal with this coordinate. Indeed, conjugating 
$(x,0,0)$ by $(0,0,1)^k$ yields the element $(e^{-k}x,0,0)$. 
This lies in $S$ for $k = \lceil\log(1+|x|) \rceil$, 
hence $|x|_S \leq 2k+1 \preceq \log(1+|x|)$. By using negative powers of $(0,0,1)$, 
the same works for $y$.

Having established these two rough inequalities completes the proof.
\end{proof}

\subsection{Hawaiian earrings and spaces of covering dimension $1$}

\subsubsection{The Hawaiian earring space and its fundamental group} 
\label{subsec:hawaiian_earring}

Recall that we denote by $\Earr$ the Hawaiian earring space, meaning the union of countably many 
circles that are all tangent to a given basepoint and whose radii tend to 
zero. Note that since the 
lengths of the circles tend to zero, the fundamental group $\pi_1(\Earr)$ is \emph{much larger} 
than the free group on countably many generators. Indeed, it also contains loops that go around 
infinitely many circles, as long as each circle appears only a finite number of times ; this can be 
made precise as follows.

By collapsing all but finitely many circles onto the basepoint, the Hawaiian earring retracts 
onto finite wedge of circles. 
This defines, for every $n\geq 1$, a surjective morphism from $\pi_1(\Earr)$ to the free group 
$F_n=\langle c^{(1)}, \dots, c^{(n)} \rangle$, where $c^{(l)}$ is a generator of the fundamental 
group of 
the $\ell$-th circle.
The restriction of the retraction induces also a surjective homomorphism $F_n \to F_{n-1}$ for 
every 
$n\geq 2$, which sends $c^{(n)}$ to $1$. This makes the groups $(F_n)_{n\geq 1}$ into a projective 
system, and we get a natural homomorphism $\pi_1(\Earr) \to \varprojlim_n F_n$.
Recall that this projective limit is the subgroup of $\prod_{n \geq 1} F_n$ consisting of sequences 
$(f_n)_{n \geq 1}$ where, for every $n\geq 2$, the element $f_{n-1}$ is the image of $f_n$ under 
the natural 
map $F_n \to F_{n-1}$.
This discussion leads to the following characterization, which will be of use in Section 
\ref{sec:change_scale}.

For $\ell\geq 1$ and an element $x \in F_n$, define $w_{\ell}(x)$ to be the number of 
occurrences of $c^{(\ell)}$ and of its inverse in the reduced word representation of $x$.
\begin{Thm}[\cite{morgan1986van}] \label{thm:description_pi_1_Earr}
	The map $\pi_1(\Earr) \to \varprojlim_n F_n$ is injective, with image the subgroup 
	\[
		\left\{ 
		(f_i) \in \varprojlim_{n \geq 1} F_n : 
		\text{for all } \ell \geq 1, \text{ the sequence } 
	 	(w_{\ell}(f_n))_{n\geq 1}
		\text{ is bounded}
		\right\}.
	\]
\end{Thm}

See \cite{desmit1992fundamental} for a short proof of this result. In fact, elements of 
$\pi_1(\Earr)$ are precisely the \emph{transfinite words} 
over the countable alphabet $\{c^{\ell} :\ell \geq 1\}$, see \cite[Thm. 
3.8]{cannon2000combinatorial}.

Using the previous theorem, one shows the following.

\begin{Cor}[see \cite{desmit1992fundamental}]
	The group $\pi_1(\Earr)$ is uncountable, nonfree, and
	contains a free group on uncountably many generators.
\end{Cor}

To see that it contains a free group on uncountably many generators, take $(I_x)_{x \in \R}$ a 
continuum of infinite subsets of $\N$ such $I_x\cap I_y$ is finite for $x\neq y$:
for example, identify $\N$ with $\Q$ and take $I_x$ to be an infinite subset having $x$ as its 
unique accumulation point. 
Then, for $x \in \R$, the elements $u_x=\prod_{\ell \in I_x} c^{(\ell)}$, 
where the product is done in increasing order, generate a free group in $\pi_1(\Earr)$. 
Indeed, given distinct $x_1,\dots,x_n$ in $\R$, 
choose $\ell_j \in I_{x_j} - \bigcup_{k\neq j} I_{x_k}$ and retract $\Earr$ onto the union of the 
circles $c^{(\ell_1)}, \dots c^{(\ell_n)}$. This maps $u_{x_j}$ onto the circle $c^{(\ell_j)}$, 
showing that $u_{x_1}, \dots, u_{x_n}$ generate a free group. For a much more general proof of 
this, see  {\cite[Thm 5.3]{cannon2000combinatorial}}.

\begin{Rq}
	Although $\pi_1(\Earr)$ is not free, it is locally free, meaning that all of its finitely 
	generated subgroups are free. This is a consequence of the fact that inverse limits of free 
	groups are locally free, see \cite[Thm 2.5]{cannon2000combinatorial}.
\end{Rq}

\begin{Rq}
	The first homology group of the Hawaiian earrings, namely the abelianization of $\pi_1(\Earr)$,
	can also be computed using the theory of algebraically 
	compact abelian groups, see \cite[Thm 3.1]{eda2000singular}. 
	But the resulting isomorphism is not explicit and will not be of use here.
	However, $\Ho_1(\Earr)$ contains an explicit copy of the Baer-Specker group $\Z^\N$, made out 
	of 
	the image of the loops which have a non-zero winding number around at least one circle. This 
	subgroup plays a role in Section \ref{sec:change_scale} as our source of big homology groups 
	(see Theorem \ref{thm:main_homology}).
\end{Rq}

Our goal is to produce homotopically embedded copies of $\Earr$. Let us present the tools we use to 
achieve this.

\subsubsection{Covering dimension} \label{subsec:cov_dim}

We rely on dimension arguments, namely covering dimension. This topological invariant is defined, for a normal topological space $X$, as the smallest integer $n$ such that every open cover of $X$ has a refinement of 
order at most $n+1$. 
Here, the \emph{order} of an open cover $\mathcal{U}$ is the largest number (possibly infinite) of 
sets in $\mathcal{U}$ that have nonempty intersection. 
For example, the euclidian space of dimension $n$ has covering dimension $n$, and ultrametric 
spaces have covering dimension $0$.
For more on this, see \cite{engelking1978dimension}.

We make use of the following result, due to Burillo.

\begin{Thm}[{\cite[Section 7.3]{burillo1999dimension}}] \label{thm:dim_cone_SOL}
	Every asymptotic cone of $\SOL$ has covering dimension $1$.
\end{Thm}

This comes from the fact that the $xy$ plane in $\SOL$ is exponentially distorted (it is the exponential radical). This means that the restriction of the metric to this plane is a so-called $\log$-metric, of the form $\log(1+d)$. But one can see that asymptotic cones of $\log$-metrics are ultrametrics, which have covering dimension zero. 
Roughly speaking, only the $t$ direction has positive dimension in the asymptotic cone, hence the result. See \cite[Prop. 11, Section 6]{burillo1999dimension} for more on this.

\begin{Rq}
	Building upon Burillo's ideas, Cornulier shows in \cite{cornulier_dimension_2008} that for every simply connected solvable Lie group $G$, the space $\cone[G]$ has covering dimension bounded above\footnote{
		\cite{cornulier_dimension_2008} actually proves that these two dimensions are equal, but 
		the reverse inequality is not related to Burillo's work.}
	 by $\dim(G/\Rexp)$. This shows that while asymptotic cones of Lie groups might fail to be locally compact, they still remain ``not too big'', in the sense that they are finite dimensional.
\end{Rq}

We make use of Theorem \ref{thm:dim_cone_SOL} by using the following characterization. Spaces of covering dimension $1$ 
are precisely the normal topological spaces where
maps from closed subspaces to the circle $\S^1$ can always be globally extended (see \cite[Thm 
3.2.10]{engelking1978dimension}).

Building upon this characterization and using Tietze's extension theorem, Burillo establishes that 
maps from closed subspaces to the Hawaiian earrings can always be globally extended in 
spaces of covering dimension $1$. Moreover, Burillo's proof can be adapted to the case of a finite 
wedge of circles. Hence the following.

\begin{Thm}[{ \cite[Thm 3.2.10]{engelking1978dimension}, \cite[Thm. 19]{burillo1999dimension}}]
	\label{thm:extension_towards_Earr_dim1}
	Let $X$ be a metric space of covering dimension $1$ 
	and let $E$ be either $\S^1$, $(\S^1)^{\vee k}$ or $\Earr$. 
	Then for every closed subset $A$ of $X$ and every continuous map 
	$f : A \to E$, there exists a continuous extension $F:X \to E$.
\end{Thm}

The key consequence here is that this property entails injections in homotopy, as follows.
A classical source of injections in homotopy and homology is retracts. 
Recall that a subspace $A \subset X$ is a retract if the inclusion $i$ of $A$ in $X$ has a left 
inverse, meaning a map $r:X\to A$ such that $r\circ i=\operatorname{Id}_A$. 
A weakening of this notion, which we call \emph{retract up to homotopy}, is for the inclusion to 
have a left inverse up to homotopy, meaning this time that $r \circ i $ is homotopic to 
$\operatorname{Id}_A$; this still yields that $i$ induces injections in homotopy and homology.

The following is a slight restatement of a result of Burillo; we give its 
elementary proof for the sake of completeness.

\begin{Cor}[After {\cite[Cor. 20]{burillo1999dimension}}] \label{cor:embedded_covdim_one}
	Let $X$ be a space of covering dimension $1$. If $ i: A \xhookrightarrow{} X$ is a closed 
	subset 
	homotopically equivalent to the circle $\S^1$, to a finite wedge of circles $(\S^1)^{\vee k}$ 
	or to the Hawaiian earring $\Earr$, then $A$ is a retract up to homotopy.
	In particular, the induced map $i_*$ is injective in homotopy and in homology. 
\end{Cor}

\begin{proof}
	Let $i: A\hookrightarrow X$ be the inclusion of a closed subset. Let $f:A\to E$ be a 
	homotopy equivalence, where $E$ is either $\S^1$, $(\S^1)^{\vee k}$ or $\Earr$.
	
	By Theorem \ref{thm:extension_towards_Earr_dim1}, one can extend $f$ 
	to a continuous map $F:X\to E$ such that $F\circ i= f$. 
	Let $g:E \to A$ denote an homotopy inverse of $f$. Then $r=g\circ F : X \to A$ is such that 
	$r\circ i = g\circ F \circ i = g\circ f$ is homotopic to $\operatorname{Id}_A$, hence $A$ is a 
	retract up to homotopy.
	Now $r_*\circ i_*=\operatorname{Id}$ in homotopy and homology, so $i_*$ is injective.
\end{proof}

This result a key tool to find inclusions between fundamental groups. Indeed, it is enough
to \emph{topologically embed} $\Earr$ into $\cone$ 
in order to find inclusions $\pi_1(\Earr)< \pi_1(\cone)$ and $\Ho_1(\Earr) < \Ho_1(\cone)$.

\begin{Rq}
	For general $G$, the asymptotic cone $\cone[G]$ is not of covering dimension $1$, so Theorem 
	\ref{thm:extension_towards_Earr_dim1} does not directly 
	apply. What we do here instead, given a $\SOL$ quotient, is to construct loops inside 
	the $1$-dimensional asymptotic cone of $\SOL$ and prove that their lifts to $\cone[G]$ are 
	still loops.
	
	The issue is that injectivity of the base loops is lost when making their lifts be loops. 
	So we have to compute precisely where overlap occurs, first to obtain embeddings in $\cone$, 
	and second to show that the degeneration map remains injective in homology and homotopy.
\end{Rq}

\section{Setup of proof} \label{sec:setup}
In this section, we define the sequence of words that we wish to evaluate and use it to give a 
detailed statement of our main result, namely Theorem \ref{thm:technical_main}.

\subsection{The words $\Theta_k$} \label{subsec:def_words}
Let us start by defining our main tool, namely words to evaluate into groups.

We consider the free group on two generators $F=F_{U,Y}$, meaning that we study reduced words in 
the letters $U,Y$ and their inverses. We call these \emph{group words}, in contrast with 
\emph{positive words} where only positive powers of the letters appear. Given such a group word 
$w$, we often write $\overbar{w}$ in place of $w^{-1}$.

In this context, we denote the commutator of group words as

\[
\lb v,w\rb \defeq vw\overbar{v} \overbar{w}. 
\]

\begin{Def} \label{def:words_Theta}
We define a sequence $(\Theta_k)$ of group words in $F_{U,Y}$ by the induction formula
\begin{align*}
	&\Theta_0 = U; 
	&&\Theta_k = \lb\Theta_{k-1},Y\rb. 
\end{align*}
\end{Def}

The point is that the words $\Theta_k$, when evaluated on group elements, will allow us to define 
loops in $\SOL$ and their lifts in $G$ at the same time. We iterate commutators in order to 
eventually get back to the identity when evaluating the $\Theta_k$'s on elements of the exponential 
radical, which is nilpotent.

\subsection{Detailed statement of the main result} 
Let us now gather our ingredients. Recall that we fix a non-principal ultrafilter $\omega$, and denote by $\SOL$ the solvable Lie group $\R^2 \rtimes_{(1,-\alpha)} \R$, where $\alpha>0$. In the sequel, we fix a real triangulable Lie group $G$ with the SOL obstruction, meaning there is a continuous surjective homomorphism
\[
p: G \to \SOL.
\]
This induces a continuous surjection 
\[
p_{\omega} : \cone[G] \to \cone
\] 
between asymptotic cones.

Consider now the exponential radical $H$ of $G$. It is a nilpotent normal subgroup whose image in $\SOL$ contains its exponential radical $\R^2$ (see Section \ref{subsec:exp_rad} and Lemma \ref{lem:image_exp_rad}).
Hence if we denote by $x,y,t$ the group elements $(1,0,0), (0,1,0)$ and $(0,0,1)$ in $\SOL$, we can lift them to elements $\tilde{x},\tilde{y}$ in $H$ and $\tilde{t}$ in $G$.

Let $(\Theta_k)$ be the sequence of group words over $U,Y$ from Definition \ref{def:words_Theta}.
For $n,l\geq 1$, let $\theta_{k,n}^{(\ell)}$ be the loop in $\SOL$ defined as the evaluation (see 
Definition \ref{def:evaluation_mots_elements})
\[
\theta_{k,n}^{(\ell)}=\lev \theta_{k}(U=t^{n/\ell}xt^{-n/\ell}, Y=y ) \rev[\SOL],
\]
and let $\theta_{k}^{(\ell)}$ be the loop in $\cone$ obtained as the ultralimit of 
$(\theta_{k,n}^{(\ell)})_{n\geq 0}$. Crucially, those are indeed loops since $x$ and $y$ commute 
and $\sg$ normalizes $\sg[x]$.

With this setup in mind, we have the following result, which readily implies Theorem \ref{thm:main}.

\begin{Thm} \label{thm:technical_main} 
	Let $G$ be a real triangulable Lie group with the SOL obstruction, and let $p : G \to \SOL$ be 
	a surjective homomorphism. Let $H$, $x,y,t$ and $\widetilde{x}, \widetilde{y},\widetilde{t}$ be 
	as above, and let $k\geq 0$ be such that $H$ is $k$-step nilpotent.
	\smallskip
		
	The sequence of loops $ (\theta_{k}^{(\ell)})_{l\geq 1}$  
	in $\cone$ defines a continuous map $\Phi : \Earr \to \cone$;
	the paths $(\widetilde{\theta}_{k}^{(\ell)})_{l\geq 1}$ obtained by replacing $x,y,t$ with 
	their lifts are loops as well, and define similarly a continuous map $\widetilde{\Phi} : \Earr 
	\to \cone[G]$ 
	such that $p_{\omega} \circ \widetilde{\Phi} =\Phi$.

	Furthermore, $\Phi$ is $\pi_1$-injective, hence so is $\widetilde{\Phi}$.
\end{Thm}

The map $\Phi$ will also induce an inclusion in homology, not of the whole of $\Ho_1(\Earr)$ but 
still of a subgroup isomorphic to $\Z^\N$, yielding Theorem \ref{thm:main_homology}.

The remainder of this paper is entirely devoted to the proof of Theorem \ref{thm:technical_main}, 
up to the end of Section \ref{sec:change_scale} which deals with homology to obtain Theorem
\ref{thm:main_homology}.
We start by studying 
the combinatorial properties of the words $\Theta_k$ in Section \ref{sec:combinatorial}. We then 
study the geometric properties of their evaluations into $\SOL$ in Sections \ref{sec:geom_loops} to 
\ref{sec:change_scale}. 

\section{The combinatorial study} \label{sec:combinatorial}

\subsection{The words $\Theta_k$ in the free group} 

This section is mostly independent of the preliminaries of Section \ref{sec:preliminaries}, and 
studies the combinatorics of the words $\Theta_k$ in the free group $F_{U,Y}$ given by the 
recursive formula $ \Theta_0 = U$ and $\Theta_k = \lb\Theta_{k-1},Y\rb$.
 
\begin{Ex}
	The first words $\Theta_{k}$ are 
	\begin{align*}
		&\Theta_0 = U ; \\
		&\Theta_{1} =  UY\overbar{U}\overbar{Y}; \\
		&\Theta_{2} =  \Theta_1 Y  \overbar{\Theta_1} \overbar{Y} = UY\overbar{U} 
		YU\overbar{Y}\overbar{U} \overbar{Y}; \\
		& \Theta_{3} = \Theta_2 Y  \overbar{\Theta_2} \overbar{Y} 
					 = UY\overbar{U} YU\overbar{Y}\overbar{U}YUY\overbar{U}\overbar{Y}U\overbar{Y}\overbar{U} \overbar{Y}.
	\end{align*}
\end{Ex}

\begin{Lem} \label{lem:combinatorial_length}
	For every $k\geq 1$, the (reduced) group word $\Theta_{k}$ has length $2^{k+1}$, and contains 
	$2^k$ occurrences of $U^{\pm 1}$ and $2^k$ of $Y^{\pm 1}$. 
	
	Moreover, it contains no occurrences of $U^{\pm2}$.
\end{Lem}

\begin{proof}
	We prove by induction that, for every $k\ge 1$, the word $\Theta_k$ can be written as
	\begin{align} \label{eq:words_Theta}
		\Theta_{k} = U N_k \overbar{U} \overbar{Y}
	\end{align}
	where $N_k \in F_{U,Y}$ contains no occurences of $U^{\pm2}$ and is such that there is no cancellation in $UN_k\overbar{U}$.
	
	This holds for $k=1$ with $N_1=Y$.  Assuming it holds for $k\geq 1$, we get by the recursive 
	definition of $\Theta_{k+1}$ (see Definition \ref{def:words_Theta}) that
	\[
		\Theta_{k+1} = \Theta_k Y \overbar{\Theta_k} \overbar{Y}
		= (U N_k \overbar{U} \underbrace{\overbar{Y}) Y(Y}_{=Y} U \overbar{N_k} \overbar{U}) \overbar{Y} = U N_{k+1} \overbar{U} \overbar{Y},
	\]
	where the sequence of words $(N_k)$ is defined by the inductive formula
	\begin{align} \label{eq:words_N}
		&N_1=Y, &&N_{k+1}= N_k \overbar{U} Y U \overbar{N_k}.
	\end{align}
	By the induction assumption $UN_k\overbar{U}$ is reduced, so the above expression of $N_{k+1}$ 
	is as well, and so is $UN_{k+1}\overbar{U}$.
	
	Now, we see by Equation \ref{eq:words_N} that for $k\geq2$, the word $N_k$ 
	starts with $Y$ and ends with $\overbar{Y}$. Hence by Equation \ref{eq:words_N} again, for all 
	$k\geq 1$ the word $N_k$ contains no occurences of $U^{\pm2}$, and by Equation 
	\ref{eq:words_Theta} 
	neither does $\Theta_k$.
	
	From the inductive formula for $N_k$ it follows that $|N_{k+1}|=2|N_k|+3$, hence by Equation 
	\ref{eq:words_Theta} we have $|\Theta_{k+1}|=|N_{k+1}|+3=2|N_k|+6=2|\Theta_k|$. Since 
	$|\Theta_1|=4$, 
	we get indeed that $|\Theta_k|=2^{k+1}$.
	
	Furthermore, the only cancellation happening at each step is one instance of $\overbar{Y}Y$. So 
	if we denote by $u_k$ (resp. $y_k$) the number of occurrences of $U^{\pm 1}$ (resp. $Y^{\pm1}$) 
	in $\Theta_{k}$, we get the induction formulas
	$
		u_k=2u_{k-1} ; ~ y_k =2y_{k-1} + 2 -2= 2y_{k-1}.
	$
	Since $u_1=y_2=2$, we have indeed that $u_k=y_k=2^{k+1}$.
\end{proof}

We need the following refinement of Lemma \ref{lem:combinatorial_length} in order to calculate the length of the evaluation of our words.

\begin{Lem} \label{lem:comb_length_bis}
	Consider the free group $F_{x,y,t}$ on three generators. Then for every $k\geq 1$ and every 
	$n\geq 0$, the reduced word $\Theta_k(U=t^nxt^{-n},Y=y)$ has length $2^{k+1}(n+1)$. 
\end{Lem}

\begin{Rq}
	Here $\Theta_k(U=t^nxt^{-n},Y=y)$ denotes the evaluation, 
	not as path (hence no braces) 
	but as a word in $F_{x,y,t}$, 
	meaning the image of $\Theta_k \in F_{U,Y}$ under the homomorphism $F_{U,Y} \to 
	F_{x,y,t}$ that sends $U$ to $t^nx^{-n}$ and $Y$ to $y$.
\end{Rq}

\begin{proof}
	By Lemma \ref{lem:combinatorial_length}, $\Theta_k$ contains $2^k$ occurences of each of 
	$U^{\pm1}$ and $Y^{\pm1}$. Hence before any cancellation, the word $\Theta_k(U=t^nxt^{-n},Y=y)$ 
	has length equal to $2^{k}(2n+1)+2^{k}=2^{k+1}(n+1)$. 
	
	So the point is that there is no cancellation when evaluating $\Theta_k$ on $(t^nxt^{-n},y)$. 
	Since $\Theta_k$ is already reduced over $(U,Y)$, cancellations can only come from occurences 
	of $U^{\pm2}$ that lead to powers of $t$ cancelling. By the second part of Lemma 
	\ref{lem:combinatorial_length}, this does not happen in $\Theta_k$. 
\end{proof}

\subsection{The symbols $A_i$} \label{subsec:building_blocks}
In order to better understand these words as $k$ grows, we need to introduce more elaborate basic 
words. 

The first word $\Theta_1$ and its conjugates by $Y$ will act as an alphabet over which the 
$\Theta_k$ are positive words, so we introduce the following.

\begin{Def} \label{def:words_A}
	For $i \geq 0$, we denote by $A_i$ the element of $F_{U,Y}$ defined by 
	\[
	A_i \defeq Y^i \Theta_1^{(-1)^{i}} Y^{-i} = 
	\begin{cases}
		Y^{i} U Y \overbar{U} {\overbar{Y}}^{i+1}, & i \text{ even}; \\
		Y^{i+1} U\overbar{Y}\overbar{U}{\overbar{Y}}^{i}, &i \text{ odd}.
	\end{cases}
	\]
	We view $\A=\{ A_i : i \in \N\}$ as a new alphabet, with the $A_i$ as abstract symbols.
\end{Def}

Explicitly, we have
\begin{align*}
	&\Theta_1 = A_0;    \\
	&\Theta_2 = \Theta_1 Y \overbar{\Theta_1} \overbar{Y} = A_0 A_1   ; \\
	&\Theta_3 = \Theta_2 {(Y\Theta_2\overbar{Y})}^{-1} = A_0A_1(A_1^{-1}A_2^{-1})^{-1}= A_0 A_1 A_2 
	A_1; \\
	&\Theta_4 = \Theta_3 {(Y\Theta_3\overbar{Y})}^{-1} = A_0 A_1 A_2 A_1 A_2 A_3 A_2 A_1.
\end{align*}

This leads us to the expectation that each $\Theta_{k}$ can be rewritten as a \emph{positive} word 
over $\A$. 
Let us make this more precise. \smallskip

Let $\Sc=\{s_i :  i \in \N\}$ be an abstract alphabet. 
If $w$ is a positive word over $\Sc$, we denote by $w'$ the word obtained by writing $w$ in reverse order and shifting all indices by $1$ (meaning we replace each occurence of $s_i$ by $s_{i+1}$). For example, if $w=s_0s_7$ then $w'=s_8s_1$. 
	
\begin{Def} \label{def:words_w}
	We let $(w_k)_{k\ge 1}$ be the sequence of positive words over $\Sc$ defined by the induction formula
	\[
		w_1=s_0, \; w_{k+1}=w_k w_k'.
	\]
\end{Def}

\begin{Ex} \label{ex:words_w}
	The first few $w_k$'s are given by
	\[
		w_1=s_0, ~ w_2=s_0s_1, ~ w_3=s_0s_1s_2s_1, ~ w_4=s_0s_1s_2s_1 s_2s_3s_2s_1.
	\]
\end{Ex}

Notice from the definition that the set of symbols of $\Sc$ appearing in $w_k$ is equal to $\{s_0, 
\dots s_{k-1}\}$.
Hence we write $w_k(s_0,\dots s_{k-1})$. 

\begin{Prop} \label{prop:word_Theta_k}
	For every $k\geq 1$, the word $\Theta_k$ is equal to the evaluation of the positive word $w_k$ 
	on the alphabet $\A$, meaning
	$\Theta_k=w_k(A_0,\dots A_{k-1})$. 
\end{Prop}
Again, here it is the classical evaluation of words as words, not as paths: we only deal with
combinatorial properties in this section.

\begin{proof} 
	Since $\Theta_1=A_0=w_1(A_0)$, the equality holds for $k=1$.
	We then proceed by induction. Let $k\geq 1$ be such that $\Theta_k=w_k(A_0,\dots A_{k-1})$. By 
	definition,
	\[
		\Theta_{k+1} = \Theta_{k} {Y\Theta_{k}^{-1}\overbar{Y}}.
	\]
	Recall that we consider the symbols $A_i=Y^i \Theta_1^{(-1)^{i}} Y^{-i}$, so that we have the 
	key relation
	\[
		YA_i^{-1}\overbar{Y} = A_{i+1}. 
	\]
	Hence obtaining $Y\Theta_{k}^{-1}\overbar{Y}$ from $\Theta_k$ consists in reversing the order 
	and shifting all indices by one, so that $Y\Theta_{k}^{-1}\overbar{Y}$ is equal to the word 
	$w_k'$ evaluated on $\A$.
	Therefore, we have that 
	\[
		\Theta_{k+1} =w_k(A_0,\dots A_{k-1}) w_k'(A_0,\dots A_{k-1}) = w_{k+1}(A_0,\dots A_{k}).
		\qedhere
	\]
\end{proof}

Some cancellation occurs inside the evaluation $w_k(A_i)$, but the following result ensures that it 
is of no significance for the asymptotic geometry of loops studied in Section \ref{sec:geom_loops}.

\begin{Lem}[Cancellation length in $w_k(A_i)$] \label{lem:cancellation_length_wk}
	For every $k\geq 1$, the total cancellation length in the evaluation of the word $w_k$ on the family $(A_0,\dots A_{k-1})$ in $F_{U,Y}$ is bounded by $k2^{k}$. Moreover, cancellations appear only between powers of $Y$.
\end{Lem}

Here, by \emph{total cancellation length} we mean the sum of cancellations lengths (in the sense of 
Definition \ref{def:cancellation_length}) of all consecutive pairs of symbols from $\A$ in the word 
$w_k(A_0,\dots, A_{k-1})$, when viewing these symbols as elements of  $F_{U,Y}$.

\begin{proof}
	From Definition \ref{def:words_w}, see that $w_k$ has length $2^{k-1}$ over $\Sc$ and is made 
	out of letters whose indices are consecutive, meaning that its factors of length $2$ are all of 
	the form $s_is_{i\pm1}$. 
	The words $A_i$ are reduced over $U,Y$ (see Definition \ref{def:words_A}),
	so the only cancellations occurring in the evaluation $w_k(A_i)$ are between concatenations of 
	some $A_i$ with $A_{i\pm1}$, with $i\leq k-1$. 
	
	If for instance $i$ is even, $A_i$ ends with $U^{-1}Y^{-(i+1)}$ and $A_{i\pm1}$ starts with $Y^{i+1\pm1}U$, so there is a cancellation $Y^{-(i+1)}Y^{i+1 \pm 1}=Y^{\pm 1}$ between powers of $Y$. In particular, the cancellation length of $A_i$ with $A_{i\pm1}$ is equal to $2i+2$.
	The case where $i$ is odd is similar, with a cancellation $Y^{-i}Y^{i\pm1}=Y^{\pm1}$ of length 
	$2i$ between powers of $Y$. Since $i\leq k-1$, this length is always bounded by $2k$.
	
	Now $w_k$ has length $2^{k-1}$ so it contains $2^{k-1}-1$ consecutive pairs of symbols.
	Hence the total cancellation length is bounded by $2k(2^{k-1}-1) \leq k2^k$. Moreover, 
	cancellations indeed only happen between powers of $Y$. 
\end{proof}

\subsection{The building blocks $B_j$} \label{subsec:arcs_B_j}

As we wish to study distances between loops induced by the $A_i$'s, we introduce one final set of 
symbols to help through computations. This allows to split each $A_i$ into two pieces, 
highlighting the parts where different $A_i$ overlap.

We define the words $B_j$ for $j\geq 0$ by
\[
	B_j = Y^j \Theta_0 {\overbar{Y}}^{j}.
\]

Note that contrary to the definition of $\A$, we do not put negative powers of $\Theta$ inside, and 
use $\Theta_0$ instead of $\Theta_1$. From the definition, we have the following.

\begin{Prop} \label{prop:decomposition_A_in_B}
	For every $i\geq 0$, we have
	\[ \\
		A_i = ({B_i\overbar{B_{i+1}}})^{(-1)^i} =
		\begin{cases}
			& B_i \overbar{B_{i+1}}, \text{ if } i \text{ is even}; \\
			& B_{i+1}\overbar{B_{i}},\text{ if } i \text{ is odd}.
		\end{cases}
	\]
	Moreover, the cancellation length of this product in $F_{U,Y}$ is equal to $2i$, and only 
	powers of $Y$ cancel out.
\end{Prop}

In particular, we see that $A_i$ and $A_{i-1}$ share  $\overbar{B_{i}}$ as a suffix whenever $i$ is odd, and share $B_{i}$ as a prefix whenever $i$ is even. 

\begin{proof}
	First, we have that $A_0=\Theta_1= \Theta_0 Y \overbar{\Theta_0}\overbar{Y} = B_0 
	\overbar{B_1}$. Hence $\overbar{\Theta_1}=B_1 \overbar{B_0}$, and conjugating by $Y$ we get 
	$A_1= Y \overbar{\Theta_1} \overbar{Y} =B_2 \overbar{B_1}$. 
	Since all $A_i$'s are conjugates by powers of $Y$ of either $A_1$ or $A_0$, we get the desired formulas by conjugating these equalities by $Y^i$. 
	
	For the second part see that, since $\Theta_0=U$, the only cancellation happening in the 
	concatenation of $B_i$ with $\overbar{B_{i+1}}$ is one instance of $Y^{-i}Y^{i+1}=Y$, which is 
	indeed of length $2i$. 
	\end{proof}

\section{The geometric loops} \label{sec:geom_loops}

In this section, we evaluate the words $\Theta_k$ given in Definition \ref{def:words_Theta} on 
suitable elements of the groups we are dealing with in order to produce combinatorial paths, and 
eventually loops. 

Recall that we fix a solvable Lie group $G$ with the SOL obstruction, meaning we have a surjective 
homomorphism
	$p: G \to \SOL$
for some $\alpha>0$. This induces a continuous surjection between asymptotic cones, 
denoted by $p_{\omega} : \cone[G] \to \cone$.
 Recall also that we denote by $x,y,t$ the standard basis elements in $\SOL$, namely $(1,0,0)$, 
 $(0,1,0)$ and $(0,0,1)$ respectively. Denoting by $H$ the exponential radical of $G$, we fix lifts 
 $\tilde{x},\tilde{y}$ in $H$ (given by Lemma \ref{lem:image_exp_rad}) and $\tilde{t}$ in $G$ of 
 $x,y,t$ through $p$.

From each word $\Theta_k$, we construct loops $\theta_{k,n}$ in $\SOL$  of length roughly $n$. Each 
sequence $(\theta_{k,n})_{n\geq 0}$ defines a loop $\theta_k$ in the asymptotic cone $\cone$. Now 
the point is that the lifts $\widetilde{\theta_{k,n}}$ to $G$ need not be loops for small $k$, so 
we have to fix $k$ large enough to get loops in $\cone[G]$.

\subsection{Definition of the loops $\theta_{k,n}$}

Recall from Definition \ref{def:evaluation_mots_générateurs} that given a word $r$ in the letters $x,y,t$ and their inverses, we denote by $\lev r\rev[\SOL]$ its evaluation as a path based at $1$ in (the Cayley graph of) $\SOL$; 
and similarly $\lev r\rev$ for a word in $\tilde{x},\tilde{y},\tilde{t}$ evaluated in $G$. Recall also from Definition \ref{def:evaluation_mots_elements} that given a word $R$ in letters $U,Y$ and their inverses, we may evaluate it on $U=t^nxt^{-n}$ and $Y=y$ to get a word over $x,y,t$ that we then evaluate as a path $\lev R(U=t^nxt^{-n},Y=y)\rev[\SOL]$.

\begin{Not} \label{not:evaluation_in_t,x,y}
	In what follows, our notational convention is to use capitals letters for group words over $\{U,Y\}$, say $R$, and lowercase with a subscript $n$ for their evaluation on $U=t^nxt^{-n},Y=y$ as paths in $\SOL$, say $r_n$. 
	This is what we use for $\Theta_k$ and $\theta_{k,n}$ here, as well as $A_i$, $B_j$ and 
	$a_{i,n}$, $b_{j,n}$.
\end{Not}

\begin{Def} \label{def:paths_theta}
	Let $k\geq 0$ and $n\geq 1$ be integers. We define the path $\theta_{k,n}$ to be the evaluation 
	of the word $\theta_k$ on $U=t^nxt^{-n}$ and $Y=y$ in $\SOL$. 
	 Equivalently, we have 
	\[
		\theta_{0,n} = \lev t^nxt^{-n}\rev; \; \theta_k = \lev \lb \lb \lb t^nxt^{-n},y\rb, \dots 
		\rb,y\rb \rev,
	\]
	where $k$ commutators appear in $\theta_{k,n}$.
	
	Similarly, we define the path $\widetilde{\theta}_{k,n}$ to be the evaluation of the word 
	$\theta_k$ on $U=\tilde{t}^n\tilde{x}\tilde{t}^{-n}$ and $Y=\tilde{y}$ in $G$.
\end{Def}

Since $\tilde{x},\tilde{y}, \tilde{t}$ are lifts of $x,y,t$ through the surjection $p$, the path 
$\tilde{\theta}_{k,n}$ is a lift of $\theta_{k,n}$ through $p$.

We have the following restatement of Lemma \ref{lem:comb_length_bis}.

\begin{Lem} \label{lem:length_path_theta_k,n}
	For every $k,n\geq 1$, the length of the combinatorial path $\theta_{k,n}$ is 
	\[
		{\ell}_{k,n}= 2^{k+1}(n+1). 
	\]
	In particular, it is roughly equal to $n$, with constants independant of $n$.
\end{Lem}

\begin{Rq}
	The two indices $k$ and $n$ play very different roles. On the one hand, the integer $k$ denotes 
	the order of commutators and will be fixed for a given group $G$: it is enough to take $k$ such 
	that the exponential radical is $k$-step nilpotent. On the other hand, $n$ parametrizes the 
	length of the paths and will tend to infinity in order to give rise to loops in the asymptotic 
	cone.
\end{Rq}

\begin{Rq} \label{rq:lacet_theta_1}
	We have
	\begin{align*}
		\theta_{1,n}=\lev t^nxt^{-n} y t^{n}x^{-1}t^{-n} y^{-1}\rev[\SOL]=\lev\lb 
		t^nxt^{-n},y\rb\rev[\SOL],
	\end{align*}
	with $t^nxt^{-n}=(e^n,0,0)$ and $y=(0,1,0)$ commuting in $\SOL$, so that the path 
	$\theta_{1,n}$ \emph{ is a loop in $\SOL$}. Explicitly, this combinatorial loop of length 
	${\ell}_{1,n}=4n+4$ is given by the sequence of points 
	\begin{equation}
		\begin{aligned} \label{eq:theta_{1,n}}
			\theta_{1,n} = (& 1,t, \dots t^n, \\
			& t^nx, t^nxt^{-1}, \dots t^nxt^{-n}, \\
			&t^nxt^{-n}y, t^nxt^{-n}yt, \dots, t^nxt^{-n}yt^n=yt^nx,  \\
			& t^nxt^{-n}yt^nx^{-1}=yt^n, yt^{n-1}, \dots , y ,1),
		\end{aligned}
	\end{equation}
	where we denote abusively elements of $\SOL$ as words in $x,y,t$ -- for example, $t^nxt^{-n}y=yt^nxt^{-n}=(e^n,1,0)$ in $\SOL$.
	
	Since $\theta_{1,n}=\lev\Theta_{1}(t^nxt^{-n},y)\rev[\SOL] = \lev[t^nxt^{-n},y]\rev[\SOL]$ is a 
	loop, the 
	evaluation of the inverse word $\overbar{\Theta_1}$ is simply {the reverse loop 
	$\overline{\theta_{1,n}}$}. Notice that this does not hold for paths that are not closed loops: 
	in that case, the reverse path is not based at $1$, so it is \emph{a translate} of the 
	evaluation of the inverse word.
\end{Rq}

Now, inside $G$ the elements $\Tilde{x}$ and $\Tilde{y}$ might not commute, so $\Tilde\theta_{1,n}$ 
might fail to be a closed loop. For instance, if $G$ is the group 
$\operatorname{Osc}_{\alpha}=\mathrm{Heis}_3 \rtimes_{(1,-\alpha,1-\alpha)} \R$ from Example 
\ref{ex:groups_SOL_obstruction}, 
the endpoint of $\widetilde{\theta}_{1,n}$ has an exponential coordinate in the direction of the 
commutator of $\tilde{x}$ and $\tilde{y}$ in $\mathrm{Heis_3}$.
 
However, we have the following result, which is the motivation for the construction of the higher 
order $\Theta_k$'s.

\begin{Lem} \label{lem:theta_k_is_a_loop_in_G}
	Let ${z}_{k,n}$ be the endpoint of $\widetilde{\theta}_{k,n}$ in G. For every $k\geq 0$, the 
	element $z_{k,n}$ lies inside $\mathrm{C}^{k+1}(H)$, the $(k+1)$th term in the lower central 
	series of $H=\Rexp$.
	
	In particular, for every $n$ and for $k$ such that $H$ is $k$-step nilpotent, the path 
	$\tilde\theta_{k,n}$ is a closed loop in $G$.
\end{Lem}

\begin{proof}
	Since $\Tilde{x}$ lies in $H$ which is normal in $G$, we have that $z_{0,n} \in H = 
	\mathrm{C}^1H$. Now the induction formula for $\Theta_k$ yields $ z_{k+1,n}=\lb z_{k,n},y \rb$, 
	so that for all $k$, the element $z_{k,n}$ lies in $\mathrm{C}^{k+1}H$.
	
	Since $H$ is nilpotent, $\mathrm{C}^{k+1}H$ is trivial for every $k$ greater or equal than its 
	nilpotency class $k_0$, so that $z_{k,n}=1$ and $\tilde{\theta}_{k,n}$ is a closed loop in $G$ 
	for $k\geq k_0$.
\end{proof}

Recall from Remark \ref{rq:paths_as_maps_on_segments} that we write combinatorial loops of length $\ell$ abusively as maps on $\ell \S^1$, as it does not change the loops induced between asymptotic cones.

\begin{Cor} \label{cor:words_induce_loops_in_cones}
	Let $k\geq 0$ be such that the exponential radical of $G$ is $k$-step nilpotent.
	Then for every $n \geq 1$, the paths $\theta_{k,n}$ and $\tilde{\theta}_{k,n}$ induce 
	$1$-Lipschitz maps
	 $\theta_{k,n} : \ell_{k,n}\S^1 \to \SOL$ 
	and $\tilde{\theta}_{k,n} : \ell_{k,n}\S^1 \to G$
	respectively.
	
	Therefore the ultralimits of these sequences are continuous loops $\theta_k :\S^1 \to \cone$ 
	and $\tilde{\theta}_k : \S^1 \to \cone[G]$ respectively, which satisfy the relation $p_{\omega} 
	\circ \tilde{\theta}_k =\theta_k$.
\end{Cor}

\begin{proof}
	By Lemma \ref{lem:theta_k_is_a_loop_in_G}, we have indeed that $\theta_{k,n}$ and 
	$\tilde{\theta}_{k,n}$ are loops in $\SOL$ and $G$ respectively.
	Hence we may view the loops $\theta_{k,n}$ and $\tilde{\theta}_{k,n}$ as maps from (the integer 
	points of) the circle of length $\ell_{k,n}$ to $\SOL$ and $G$ respectively.
	Now by Lemma \ref{lem:paths_are_lipschitz} these maps are $1$-Lipschitz.
	
	Consider the sequence of maps $(\theta_{k,n} : \frac{l_{k,n}}{n}\S^1 \to 
	\frac{1}{n}\SOL)_{n\geq 1}$ between rescaled metric spaces. This sequence is uniformly 
	Lipschitz, so it induces a continuous map $\theta_k$ between ultralimits.
	By Lemma \ref{lem:length_path_theta_k,n}, we have the estimate $l_{k,n} \simeq n$ with 
	constants independent of $n$. So the ultralimit of the sequence of metric spaces 
	$\frac{l_{k,n}}{n}\S^1$ is (bi-Lipschitz equivalent to) the unit circle $\S^1$.  Hence the 
	induced map is $\theta_k : \S^1 \to \cone$.
	
	The same holds with $\tilde{\theta}_{k,n}$, yielding a continuous map $\tilde{\theta}_k : \S^1 
	\to \cone[G]$. Moreover, for every $n\geq 1$ we have the equality $p\circ \tilde{\theta}_{k,n} 
	= \theta_{k,n}$. 
	So by taking ultralimits we get that $p_{\omega} \circ \tilde{\theta}_k =\theta_k$.
\end{proof}

\begin{Rq}
	The loops $\theta_{1,n}$ are the combinatorial version of the loops $\alpha_{e^{n}}$ defined in 
	\cite[Section 9]{burillo1999dimension} that give rise to nontrivial elements in 
	$\pi_1(\Cone(\SOLV))$.
	Moreover, the higher order loops $\theta_{k,n}$ are the ones considered in \cite[Section 
	12]{cornulier2017geometric}, where it is shown that they have exponential area (with respect to 
	$n$), which implies exponential growth of the Dehn function of $G$.
	
	These two results tend to indicate $\theta_k=\lim_{\omega}\theta_{k,n}$ as a candidate to be a 
	nontrivial loop in $\pi_1(\cone[G])$; this indeed follows from Section 
	\ref{sec:higher_order_loops} below.
\end{Rq}

\subsection{Image of the first sequence of loops}  \label{subsec:image_SOL}

In what follows, we denote by $d$ the metric 
induced on $\SOL$ by the left-invariant Riemannian metric given by
\[
	ds^2=dt^2 + e^{-2t}dx^2 + e^{2\alpha t}dy^2,
\]
and by $\length$ the associated length, meaning $|z|=d((0,0,0),z)$ for all $z \in \SOL$. 
To emphasize the difference with this length, we write $\length_{\R}$ for the usual absolute value on $\R$.

\begin{Conv}
	What follows involves many rough comparisons, in the sense of Definition \ref{def:rough_comparison}. From here on out, those will \emph{only involve constants independent of $n$} (but which might depend on the choice of metric). This is also the case for the quasi-isometry constants of the maps we deal with. 
\end{Conv}

\begin{Rq}[Computing the length of words] \label{rq:computing_length}
	A key ingredient in what follows is the estimate of length from Proposition 
	\ref{prop:dist_SOL}, which indeed yields constants independent of $n$. 
	We apply it to words in $t,x,y$ and their inverses viewed as elements of $\SOL$. In practice, we consider two words $w_1$ and $w_2$ and wish to compute the distance $d(w_1,w_2)$. By left-invariance of $d$, it is equal to the length $|w|$ where $w=w_1^{-1}w_2$. We then repeatedly make use of the following algebraic identities holding in $\SOL$:
	\begin{align*}
		&(a,b,c) =x^ay^bt^c, \text{ where we write } x^a \defeq (a,0,0) ; \\
		&t^nxt^{-n} = (e^n,0,0)=x^{e^n}; \\
		&t^{n}yt^{-n} = (0,e^{-\alpha n},0)=y^{e^{-\alpha n}}.
	\end{align*}
	Using the conjugation relations and commutation of $x$ and $y$, we can reduce the word $w$ to a 
	product of the form $x^{a}y^{b}t^{c}$. Then using the estimate of length from Proposition 
	\ref{prop:dist_SOL} we compute 
	\[
	|w|\simeq|c|_{\R}+\log(1+|a|_{\R}+|b|_{\R}) 
	\simeq |c|_{\R}+\log(1+|a|_{\R})+\log(1+|b|_{\R}).
	\]
\end{Rq}

We know that $\theta_{1,n}$ is a closed loop of length roughly $n$ in $\SOL$, but we can say more 
about its metric properties.

\begin{Prop} \label{prop:theta_1}
	For every $n$, the loop $\theta_{1,n}$ is a quasi-isometric embedding of the circle of length 
	${\ell}_{1,n}=4n+4$ in $\SOL$, with quasi-isometry constants independent of $n$. 
\end{Prop}

Let us now fix $n$, and denote in this section $\theta=\theta_{1,n}$ and $\pi=\theta_{0,n}$. 
Crucially, note that the loop $\theta_{1,n}$ consists of several ``$\Pi$-shaped'' paths that are 
translates of $\pi=\lev t^nxt^{-n}\rev[\SOL]$ and $\pi'=\lev t^{-n}yt^{n}\rev[\SOL]$. We can 
therefore make use of the following classical result about quasi-geodesic segments in $\H^2$, which 
is the motivation for defining the loops we use here. Recall from section \ref{subsec:SOL_alpha} 
that the planes $\sg[x]\rtimes \sg[t]$ and $\sg[y]\rtimes \sg[t]$ are totally geodesic copies of 
the hyperbolic plane in $\SOL$. 

\begin{Lem} \label{lem:quasigeod_H^2}
	The path $\pi=\lev t^nxt^{-n}\rev[\SOL]$ (resp. $\pi'=\lev t^{-n}yt^{n}\rev[\SOL]$) is a quasigeodesic segment of length $2n+1$ inside the totally geodesic plane $\sg[x]\rtimes \sg[t] \cong \H^2$ (resp. $\sg[y]\rtimes \sg[t] \cong \H^2$ ) of $\SOL$, with quasi-isometry constants independent of $n$.
\end{Lem}

\begin{Rq}
	Intuitively, this result comes from the contracting action of $t$ on $x$. As the $t$ coordinate goes up, the $e^{-2t}$ factor in the metric (see the beginning of Section \ref{subsec:image_SOL})
	contracts the $dx$ term, so that the length of a unit in the $x$ direction decreases 
	exponentially; recall the metric view of $\SOL$ from Figure \ref{fig:metric_view_SOL} in 
	Section \ref{subsec:SOL_alpha}.
	To see this geometrically, view the path $\pi=\theta_{0,n}$ as joining the 
	points $1$ and $e^nx$ by going up along a geodesic until time n, moving sideways one unit, and 
	then going down along a parallel geodesic. This forms a $\Pi$-shaped path in the upper 
	half-space model of $\H^2$ which approximates a geodesic, 
	see Figure \ref{fig:theta_1_in_coordinates}.
\end{Rq}

\begin{proof}[Proof of Lemma \ref{lem:quasigeod_H^2}]
	Notice first that the $xt$ and $yt$ planes in $\SOL$ are quasi-isometric via the map $(y,t) \mapsto (x,-\alpha t)$ (and even isometric with the standard choice of metric). This sends $\pi'$ to the path $\lev t^{n/{\alpha}}  x t^{-n/\alpha} \rev[\SOL]$, so that applying the result on $\pi$ with $n/\alpha$ instead of $n$ yields the result for $\pi'$, up to multiplying the quasi-isometry constants by $\alpha$. Hence it is enough to prove the result for $\pi$.
	
	The path $\pi=\theta_{0,n}$ consists of two geodesic segments of the form $\lev t^n\rev[\SOL]$, 
	and one segment of bounded length which can be ignored in our rough computations. Hence we only 
	need to compare distances between points of the first and the second  geodesic segment. The 
	goal is to establish the following estimate:
	\[
		d(\pi(s),\pi(s')) \simeq |s-s'|_{\R}, \text{ for all } s \in [0,n], s' \in [n+1,2n+1].
	\]
	Let $k,l \in \{0,\dots n\}$ and consider the instants $s=n-k$, $s'=n+l+1$ so that $|s-s'|_{\R} =l+k+1 \simeq l+k$. 
	By left-invariance of the metric and the estimate of length of Proposition \ref{prop:dist_SOL}, 
	we have: 
	\begin{align*}
		d(\pi(s),\pi(s')))=d(t^{n-k}, t^nxt^{-l}) = |t^kxt^{-l}| &=|t^{k}xt^{-k}t^{k-l}|
		 \\ &= |(e^k,0,k-l)| \simeq k + |k-l|_{\R}.
	\end{align*}
	Since $u+v \simeq u+|u-v|$ for nonnegative $u$ and $v$, 
	this last quantity is roughly $l+k$. 
	So $d(\pi(s),\pi(s'))$ roughly equals $|s-s'|_{\R}$
	with constants independent of $n$. This shows that $\pi$ is a quasigeodesic segment in $\H^2$ and in $\SOL$.
\end{proof}
Having this grasp on the metric properties of $\pi=\theta_{0,n}$ inside $\SOL$, we 
may now tackle $\theta_{1,n}$. In the following proof, by ``bounded'' we always 
mean bounded by a constant depending only on the choice of metric but not on 
$n$.

\begin{figure}
	\centering
	{\includegraphics[width=0.35\linewidth]{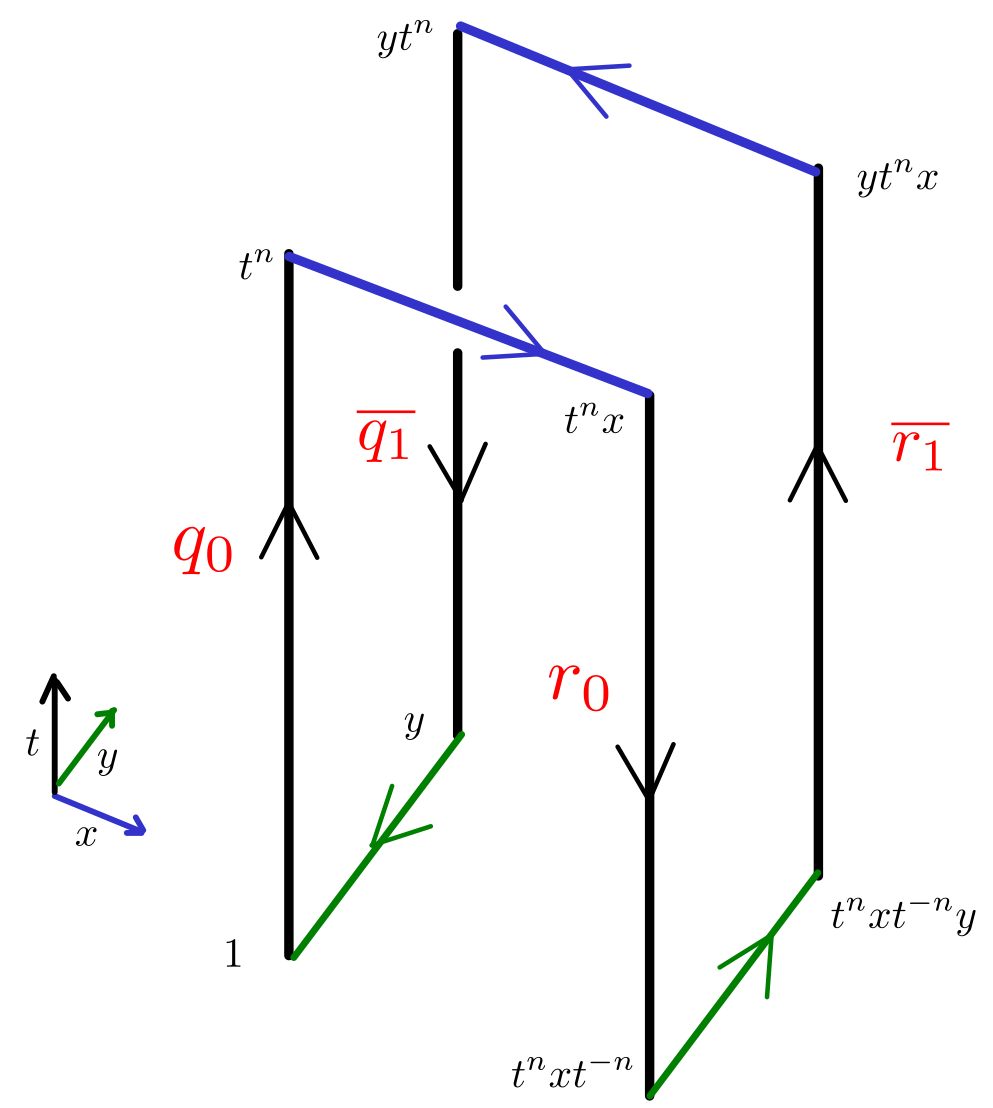}}
	\caption{The loop $\theta=\theta_{1,n}$ in coordinates $(x,y,t)$ in $\SOL$. Up to 
		parts of bounded length, it is equal to the concatenation 
		$q_0r_0\overline{q_1}\overline{r_1}$.} 
	\label{fig:theta_1_in_coordinates}
\end{figure}

\begin{proof}[Proof of Proposition \ref{prop:theta_1}]
	Recall that we fix $n\geq 1$, and that we denote in this proof
	\[
		\theta = \theta_{1,n}=\lev t^n x t^{-n} y t^n x^{-1} t^{-n}\rev[\SOL].
	\]
	This combinatorial loop in $\SOL$ consists of four parts of bounded length which we can ignore in our rough computations, and of four geodesic segments, namely translates of $\lev t^n \rev[]$. Let us denote by
	\[
		q_0 = \lev t^n \rev[], ~ r_0 =t^nx\lev t^{-n}\rev[], ~ q_1 = y\lev t^n \rev[], \text{ and } 
		r_1 = yt^nx\lev t^{-n} \rev[];
	\]
	so that the four geodesic segments of $\theta$ are (in this order) $q_0, r_0, \overline{q_1}$, 
	and $\overline{r_1}$.

	We thus need to consider two points $p,p'$ on the loop $\theta$ and compare their combinatorial 
	distance with their distance in $\SOL$. The aim is to obtain rough equality between these two 
	quantities.
	
	We claim that the case when $p$ and $p'$ lie on consecutive segments is a direct consequence of 
	Lemma \ref{lem:quasigeod_H^2}. Indeed, see in Figure \ref{fig:theta_1_in_coordinates} that both 
	concatenations $q_0r_0$ and $q_1r_1$ lie in a plane of constant $y$ coordinate, meaning a 
	totally geodesic copy of $\H^2$. Moreover, they are conjugates (by an element of bounded 
	length) of $\pi=\theta_{0,n}$, so by left-invariance of the metric they are quasi-geodesic 
	segments in their respective hyperbolic planes and hence in $\SOL$. 
	
	Now consider the concatenation $r_0\overline{r_1}$. It lies in a plane of constant $x$ 
	coordinate which is also a totally geodesic $\H^2$. Furthermore, and this last point is what 
	motivates the use of the specific path $\theta_{1,n}$ from the start, it is a left translate of 
	the path $\pi'=\lev t^{-n}yt^{n}\rev[]$. Hence by Lemma \ref{lem:quasigeod_H^2} again this 
	concatenation is also quasi-geodesic. The case of $q_0\overline{q_1}$ is similar.
	
	\begin{figure}
		\centering
		{\includegraphics[width=0.27\linewidth]{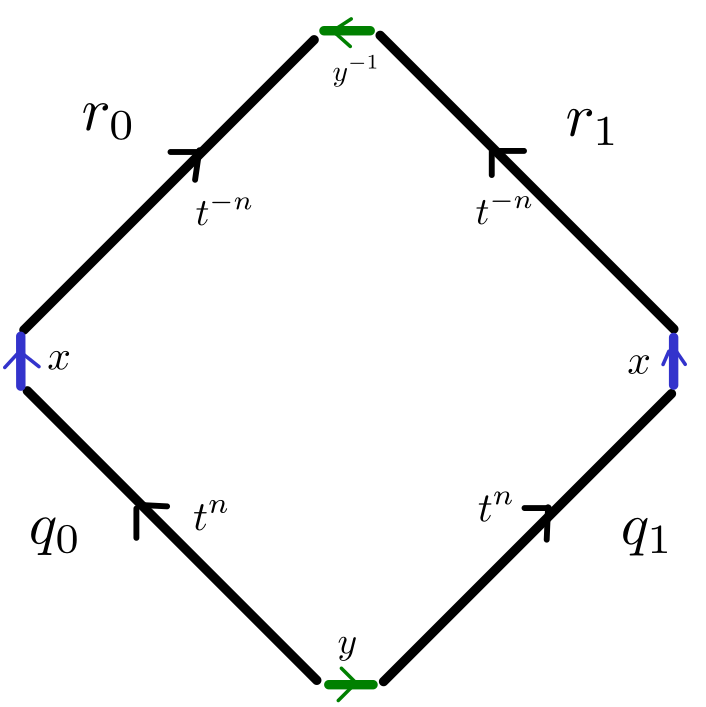}} 
		\caption{The four geodesic arcs composing $\theta=\theta_{1,n}$. Lemma 
		\ref{lem:quasigeod_H^2} implies that any concatenation of two consecutive arcs yields a 
		quasigeodesic, with uniform constants.}
	\end{figure}
	
	The only case left is when $p$ and $p'$ lie in opposite segments, say 
	\[
	p=t^k \in c_0, ~ 
	p'=yt^nxt^{-(n-l)} \in d_1
	\]
	with $0\leq k,l\leq n$. Then their combinatorial distance is roughly $n$. On the other hand, in $\SOL$ we have 
	\[
	p^{-1}p' =  t^{-k}yt^nxt^{-(n-l)}= (t^{-k}yt^k)(t^{n-k}xt^{-(n-k)}) t^{k+l} = 
	(e^{n-k},e^{\alpha k},k+l).
	\]
	By Proposition \ref{prop:dist_SOL}, the length of this element satisfies 
	\[
	|p^{-1}p'|_{\SOL} \simeq k+|n-k|+|k+l| \simeq n.
	\]
	The case $p \in c_1, p' \in d_0$ is similar.
	
	Having established the rough equality  between the combinatorial distance and the distance in 
	$\SOL$  in all cases completes the proof.
\end{proof}

\begin{Cor}
	The loop $\theta_1=\lim_{\omega}(\theta_{1,n})$ is a bi-Lipschitz embedding of the unit circle 
	into $\cone$.
\end{Cor}

\begin{proof}
	Since the exponential radical of $\SOL$ is abelian, Corollary 
	\ref{cor:words_induce_loops_in_cones} implies that $\theta_1$ is a well-defined $1$-Lipschitz 
	map from $\S^1$ to $\cone$. Moreover, Proposition \ref{prop:theta_1} states that the sequence 
	$(\theta_{1,n})_{n\geq 1}$ is uniformly quasi-isometric, so that its ultralimit is indeed 
	bi-Lipschitz (see Proposition \ref{prop:maps_induced_between_cones}).
\end{proof}

We have shown that the image of $\theta_1$ in $\cone$ is topologically a circle. In fact, it 
appears as a quadrilateral. From this combinatorial point of view, we get back the geometric loop 
$A_1$ constructed by Burillo in \cite[Section 9]{burillo1999dimension}.

\begin{figure}
	\centering
	\includegraphics[width=0.5\linewidth]{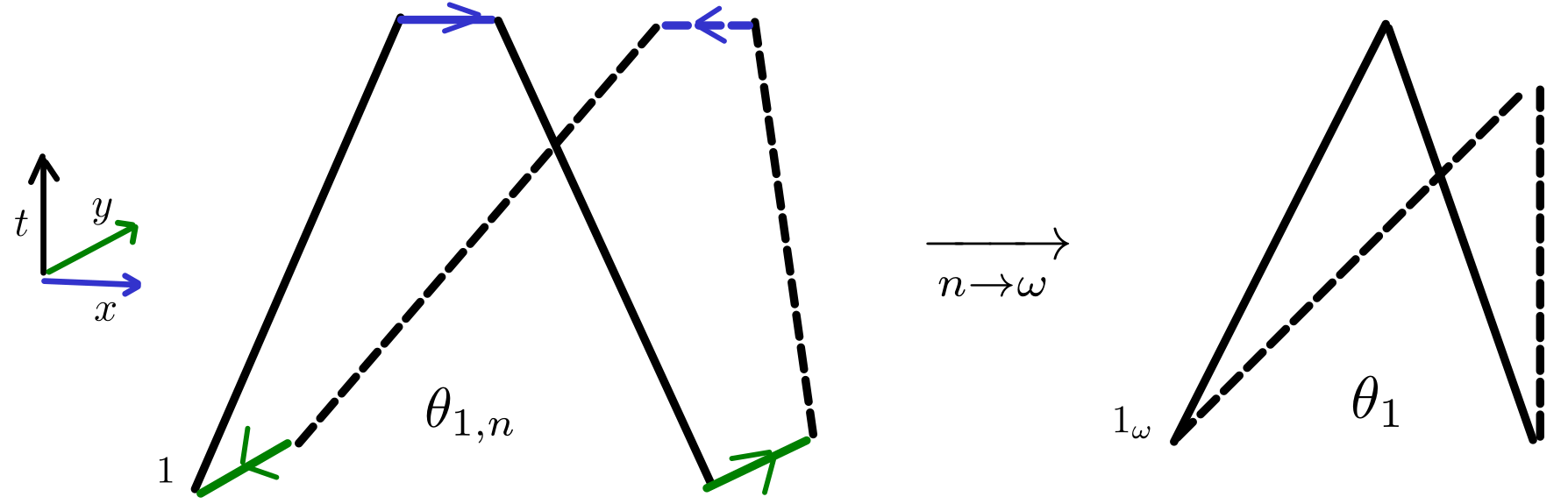}
	\caption{Metric appearance of $\theta_{1,n}$ and $\theta_1$. The bounded parts disappear in the 
	limit, leaving only $4$ segments of unit length in the cone.}
	\label{fig:lacettheta1solmetrique}
\end{figure}

\begin{Rq}
	This result of course does not hold if we replace the weight $-\alpha$ of $\SOL$ with 
	any positive value. The resulting group, as shown by Heintze in 
	\cite{heintze1974homogeneous}, admits a left-invariant metric of 
	negative curvature. In that case, the asymptotic cone is a real tree which 
	does not contain any nontrivial loop. 
	And indeed, we see that the loop $\theta_{1,n}=\lev t^nxt^{-n}yt^nx^{-1}t^{-n}y^{-1}\rev[]$ can 
	still be defined, but it is not quasi-isometric. The limit $\theta_1$ between cones simply 
	consists of two concatenated segments that are run back and forth, and is not an embedding.
\end{Rq}

Recall now that Theorem \ref{thm:dim_cone_SOL} states that $\cone$ has covering dimension $1$.
In particular, the embedded circle $\theta_1$ must be a nontrivial element of $\pi_1(\cone)$.
But, as the exponential radical of $G$ might fail to be abelian, the lifts $\Tilde\theta_{1,n}$ 
might not be loops in $G$, so this is not yet enough to find a nontrivial element in 
$\pi_1(\cone[G])$. 
This motivates the study of the loops $\theta_{k,n}$ with higher order commutators. 
The problem is that they are not embeddings anymore, hence the need for an extra factorization step.

\section{Higher order loops} \label{sec:higher_order_loops}

Recall from Definition \ref{def:paths_theta} that we define $\theta_{k,n}$ as the loop induced in 
$\SOL$ by the evaluation of the word $\Theta_{k}$ 
on $U=t^nxt^{-n}$ and $Y=y$; see Definition \ref{def:evaluation_mots_générateurs} for our 
convention of evaluations of words as paths.
We study here the geometric properties of the sequence $\theta_{k,n}$ when $k$ is fixed and $n$ 
goes to infinity.

Following our notational convention, (see Notation \ref{not:evaluation_in_t,x,y}), we denote by 
$a_{i,n}$ and $b_{j,n}$ for $i,j\geq 0$ and $n\ge 1$ the paths defined by the evaluation of the 
words $A_i$ and $B_j$ (defined in Sections \ref{subsec:building_blocks} and \ref{subsec:arcs_B_j}) 
on 
$U=t^nxt^{-n}$  and $Y=y$.

The main insight of Section \ref{sec:combinatorial} is that the words $\Theta_k$ are built out of 
the $A_i$'s, which in turn are each composed of two $B_j$'s; and this evaluates down into similar 
decompositions for the loops.
This is the key to simplify our computations, allowing us to consider the simple building blocks 
$b_{j,n}$ rather than the whole $\theta_{k,n}$.

\subsection{Embedding the circle into $\cone$}

Note that $a_{i,n}$ is a loop in $\SOL$ since $A_i$ is a conjugate of $\Theta_1^{\pm 1}$, while 
$b_{j,n}$ is simply a path.
Explicitly, we have 
\[
b_{j,n} =  \lev y^j t^nxt^{-n} y^{-j} \rev[\SOL].
\]
Since $x$ and $y$ commute, $b_{j,n}$ has fixed endpoints when $i$ varies, namely $\mathbf 1_{\SOL}$ 
and $t^nxt^{-n}=(e^n,0,0)$. 
In particular, the concatenation $b_{i,n}\overline{b_{j,n}}$ is a 
well-defined loop in $\SOL$, for which we have the following metric result that builds upon 
Proposition \ref{prop:theta_1}.

\begin{Prop} \label{prop:arcs_bj_embedding_circle}
	Let $i\neq i' \geq 0$ be integers. For every integer $n\geq 1$, the concatenation 
	$b_{i,n}\overline{b_{i',n}}$ is a quasi-isometric embedding of the circle of length roughly $n$ 
	into $\SOL$, with constants independent of $n$.
	
	Therefore, the ultralimit $b_i\overline{b_j}=\lim_{\omega}(b_{i,n}\overline{b_{i',n}},n\ge 0)$ 
	is a bi-Lipschitz embedding of the unit circle in $\cone$.
\end{Prop}

\begin{Rq}
	Notice that, by Proposition \ref{prop:decomposition_A_in_B}, we have that 
	$\Theta_1=A_0=B_0\overline{B_1}$, hence evaluating on $U=t^nxt^{-n}$  and $Y=y$ yields the 
	equality of loops $\theta_{1,n}=b_{0,n}\overline{b_{1,n}}$. So Proposition \ref{prop:theta_1} 
	is precisely the case $i=0$ and $i'=1$.
\end{Rq}
The quasi-isometry constants occurring here are independent of $n$, but they do depend on $i$ and 
$i'$ as can be seen from the proof below. However, this is not consequential, since in Section 
\ref{subsec:image_theta_k} we fix $k\geq 1$ depending only on $G$ and consider only the $b_j$'s 
where $j\leq k$.

\begin{proof} 
	 For $n\geq 1$, we consider the following paths in $\SOL$:
	\begin{align*}
		b_{i,n} = \lev y^it^nxt^{-n}y^{-i} \rev[\SOL]; \text{ and }
		b_{i',n} = \lev y^{i'}t^nxt^{-n}y^{-i'} \rev[\SOL].
	\end{align*}
	They share their endpoints $\1_{\SOL}$ and $t^nxt^{-n}$ and each have length roughly $n$, so 
	that the concatenation $\delta =b_{i,n}\overline{b_{i',n}}$ is well-defined and yields a map 
	from the circle of lenth roughly $n$ to $\SOL$.
	
	Now, up to conjugation by $y^i$, the loop $\delta$ is equal to the loop $\theta_{1,n}$ where 
	$y$ is replaced by $y^{i'-i}$. Hence the same computations as in the proof of Proposition 
	\ref{prop:theta_1} apply, and we get the result.
\end{proof}

From the combinatorial decomposition of the words $A_i$ into $B_j$'s of Proposition 
\ref{prop:decomposition_A_in_B} and the fact that the cancellation length is bounded independently 
of $n$, this Proposition implies that for every $i \geq 0$, the ultralimit 
$a_i=\lim_{\omega}(a_{i,n},n\ge 0)$ is a bi-Lipschitz embedding of the unit circle $\S^1$ in 
$\cone$. But this alone is not enough to characterize the image of the loops $\theta_k$ for 
$k\geq2$.

\subsection{Image of the higher order loop $\theta_k$} \label{subsec:image_theta_k}
We now study the loop $\theta_k$ in $\cone$. By evaluating the rewriting of the word 
$\Theta_k$ given by Proposition \ref{prop:word_Theta_k}, we see that it is equal to the positive 
word $w_k$ (from Definition \ref{def:words_w}) over loops $a_i$.
Hence the need to consider several of the loops $a_i$ at the same time. Each one of them is an 
embedding, but there is overlap between them; we make use of the decomposition of the $A_i$'s into 
$B_j$'s from Proposition \ref{prop:decomposition_A_in_B} to highlight where this lack of 
injectivity occurs. 

First, we introduce the metric space that appears as the image of our loops.

\begin{Def}
	We denote by $\Cocoa$ and call the \emph{cocoa pod with $k+1$ arcs} the metric space defined as 
	the suspension of a finite set with $k+1$ elements.
	
	Namely, it is the quotient of the disjoint union of $k+1$ copies of the unit segment denoted by 
	$b_0,\dots ,b_k$, where all the starting points are identified on one side, and all the
	endpoints are identified on the other. It is endowed with the graph metric.
\end{Def}
The following lemma is a direct consequence of the construction of the cocoa pod.

\begin{Lem} \label{lem:pi_1_cocoa_pod}
	The space $\Cocoa$ is homotopically equivalent to a wedge of $k$ circles relatively to the 
	basepoint $0$, \emph{via} a map sending $b_0$ to the basepoint and the segment $b_i$ to the 
	$i$th circle for $1\leq i\leq k$.
	In particular, $\pi_1(\Cocoa)$ is the free group over generators $(b_1\overline{b_0},\dots , 
	b_k\overline{b_0})$.
\end{Lem}

See Figure \ref{fig:cocoa_pod} for a picture of $\Cocoa[3]$, as well as the action on the labels of 
an homotopy equivalence towards a wedge of circles.

\begin{figure}
	\centering
	{\includegraphics[width=0.45\linewidth]{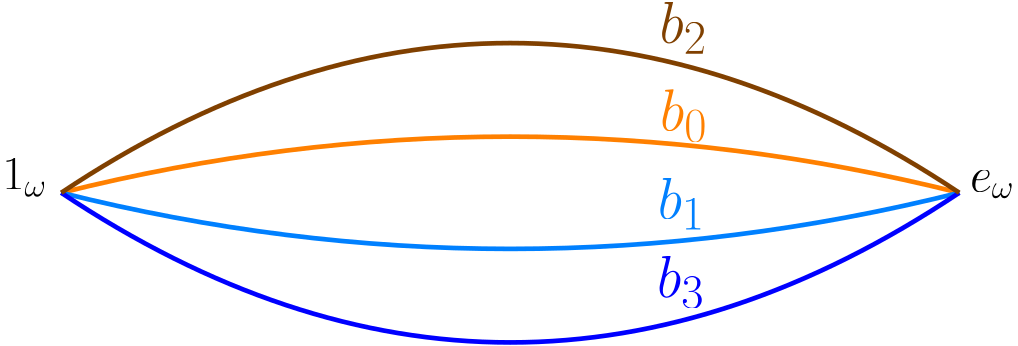}} 
	\quad
	{\includegraphics[width=0.23\linewidth]{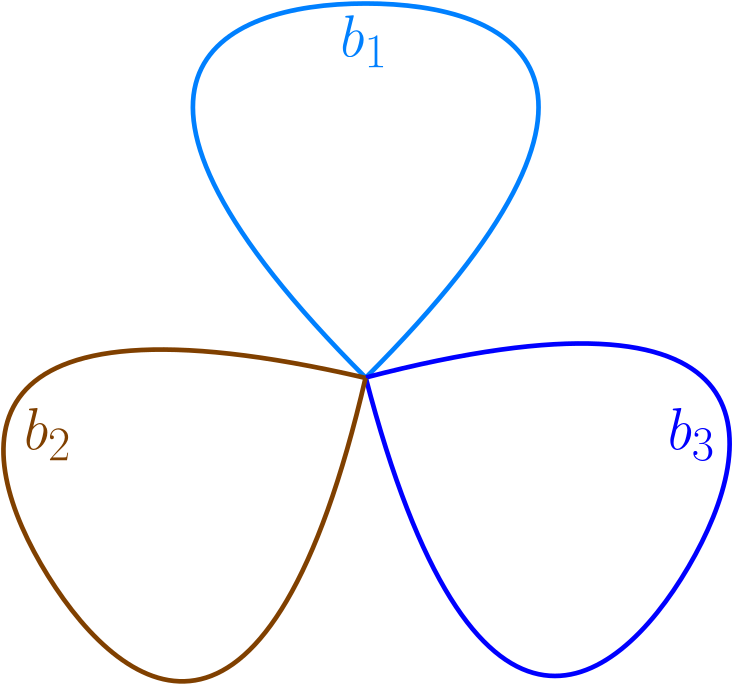}} 
	\captionof{figure}{On the left, a copy of $\Cocoa[3]$. This space is homotopically equivalent to a wedge of three circles, represented on the right with the action on the labels of an homotopy equivalence.}
	\label{fig:cocoa_pod}
\end{figure}

Because of the repetitions in the word $w_k$ and the overlap between the loops $a_i$, the loop 
$\theta_k$ is not an embedding of the circle as soon as $k\geq 2$. So we cannot use yet the 
$\pi_1$-injectivity for embedded finite wedge of circles in spaces of covering dimension $1$ as it 
is stated in Section \ref{subsec:cov_dim}.  To find back an embedding, we factor $\theta_k$ 
through a space homotopically equivalent to a wedge of $k$ circles, namely a cocoa pod.

\begin{Prop} \label{prop:image_of_theta_in_cone_is_a_wedge}
	For $k\geq 1$, the image of the loop $\theta_{k} : \S^1 \to \cone$ is a subspace of $\cone$ 
	homotopically equivalent to a wedge of $k$ circles.  
	
	More precisely, $\theta_k$ factors as a composition $\beta_k \circ W_k$ through the cocoa pod 
	$\Cocoa$, where 
	\begin{itemize}
		\item $\beta_k : \Cocoa \to \cone$ is a bi-Lipschitz embedding;
		\item $W_k : \S^1 \to \Cocoa $ is the continuous map represented by the word 
			\[
				w_k\left(A_i=(b_i\overline{b_{i+1}})^{(-1)^i}\right).
			\]
	\end{itemize}
\end{Prop}

\begin{proof}
	As in all the proofs that follow, we start by evaluating the combinatorial equalities from 
	Section \ref{sec:combinatorial} on $U=t^nxt^{-n}$ and $Y=y$ in $\SOL$ and then take the 
	ultralimit over $n$. The computations of cancellation length in Lemma 
	\ref{lem:cancellation_length_wk} and Proposition \ref{prop:decomposition_A_in_B}, and the fact 
	that these cancellations appear only between powers of $Y=y$, ensure that the total 
	cancellation length appearing is bounded independently of $n$. So the filiform vanishing of 
	Corollary \ref{cor:filiform_vanishing} applies and yields equalities between paths in 
	asymptotic cones.
	
	Applying this procedure to Propositions \ref{prop:word_Theta_k} and 
	\ref{prop:decomposition_A_in_B} yields the following equalities between loops in $\cone$:
	\begin{align}
		\theta_k &= w_k(a_0,\dots a_{k-1}); \label{eq:theta_k=w_k}\\ 
		a_i &= (b_i \overline{b_{i+1}})^{(-1)^i}.  \label{eq:a_i=b_i's}
	\end{align}
	In particular, see that the image of $\theta_k$ is equal to the union of the images of $b_0, 
	\dots b_k$.
	
	Recall that, for fixed $n\geq 1$, the paths $b_{j,n}$ share their endpoints $\1$ and $e_n=t^nxt^{-n}$ in $\SOL$ when $j$ varies. Hence denoting by 
	$
	\1_{\omega} = [\1]_{\omega}, ~ e_{\omega} =\lim_{\omega}(e_n)
	$
	the corresponding points in $\cone$, all paths $b_j$ are bi-Lipschitz embeddings of the segment 
	of unit length in $\cone$ with endpoints $\1_{\omega}$ and $e_{\omega}$. 
	
	Moreoever, when considering the $(k+1)$ first $b_j$'s together, Proposition 
	\ref{prop:arcs_bj_embedding_circle} implies that the 
	map $\beta_k : \Cocoa \to \cone$ defined by applying $b_i$ on the $i$-th segment, for $0\leq i \leq k$, is a bi-Lipschitz embedding of the cocoa pod with $k+1$ arcs.

	Our aim is now to factor $\theta_k$ through $\Cocoa$, decomposing it into first a map $W_k$ 
	that contains all the degeneration, and second the bi-Lipschitz embedding $\beta_k$ of 
	$\Cocoa$. This means that we construct the following commutative diagram:
	\[\begin{tikzcd}
		{\S^1} &&& \cone \\
		& {\Cocoa} \\
		\arrow["{\theta_k}", from=1-1, to=1-4]
		\arrow["{{W}_k}"', from=1-1, to=2-2]
		\arrow["{\beta_k}"', from=2-2, to=1-4]
	\end{tikzcd}\]
	Let $b_i$, $0\leq i \leq k$, denote the $k+1$ arcs of $\Cocoa$, so that we (for the sake 
	of simplicity) confuse the arcs with the paths that $\beta_k$ applies on each of them. As in 
	the statement of the theorem, let $W_k : \S^1 \to \Cocoa$ be the map represented by the 
	word
	\[
	w_k\left(A_i=(b_i\overline{b_{i+1}})^{(-1)^i}\right).
	\] 
	By this we mean that $W_k$ cuts the circle into $|w_k|$ equal arcs, labels them by the letters 
	of $w_k$, and sends each occurence of the letter $A_i$ to the loop 
	$(b_i\overline{b_{i+1}}^{(-1)^i})$ in $\Cocoa$, homeomorphically outside the endpoints of the 
	arc. Since all these loops are based at the initial point $\1_{\omega}$ of all segments $b_i$, 
	the map $W_k$ is continuous.
	
	By composing Equations \ref{eq:theta_k=w_k} and \ref{eq:a_i=b_i's}, we get the equality of loops
	\[
		\theta_k = w_k\left(A_i=(b_i\overline{b_{i+1}})^{(-1)^i}\right),
	\]
	which is equivalent to the equality of maps $\theta_k=\beta_k \circ W_k$.
\end{proof}

With this factorization step, we are now able to prove $\pi_1$-injectivity of the loop $\theta_k$, 
using both the dimension arguments of Section \ref{subsec:cov_dim} and the combinatorial properties 
from Section \ref{sec:combinatorial}.

\begin{Cor} \label{cor:theta_k_SOL}
	There exists a positive word $\nu$ over $k$ letters such that the following holds.
	Denoting by $c$ and $c_1, \dots, c_k$ suitable generators of  the fundamental groups of $\S^1$ 
	and of $\Cocoa$ respectively, the 
	map between fundamental groups induced by $\theta_k$ is given 
	by 
	\begin{align*}
		\theta_{k,*} : &\sg[c]\cong \Z \to \sg[c_1,\dots c_k] \cong F_k < \pi_1(\cone) \\
		&c \longmapsto \nu(c_1,\dots c_k)
	\end{align*}
	In particular, the loop $\theta_k$ is nontrivial in $\pi_1(\cone)$.
\end{Cor}

\begin{proof} 
	First, the space $\Cocoa$ is homotopically equivalent to a finite wedge of circles (see Lemma 
	\ref{lem:pi_1_cocoa_pod}), so its fundamental group is a free group $F_k$ with a 
	system of generators given by
	\[
	c_i =  (b_i \overline{b_0})^{(-1)^i}, 1\leq i \leq k.
	\]
	By the $1$-dimensionality of $\cone$ (see Theorem \ref{thm:dim_cone_SOL}) and the fact that 
	wedge of circles are retracts up to homotopy in spaces of dimension $1$ (see Corollary 
	\ref{cor:embedded_covdim_one}), the map $\beta_{k,*}$ is an injection between fundamental 
	groups. Hence we can view $F_k =\langle c_1\dots c_k \rangle$ as a subgroup of $\pi_1(\cone)$.
	Now with this choice of generators, the map $\beta_{k,*}$ acts on 
	the loops $a_i = (b_i \overline{b_{i+1}})^{(-1)^i}$ (see Equation \ref{eq:a_i=b_i's}) by
	\[
		\beta_{k,*} (a_i) = \begin{cases}
								c_1, ~ i=0; \\
								c_i c_{i+1},~ i>0 \text{ even;} \\
								c_{i+1} c_i,~ i \text{ odd}. 
							\end{cases}.					
	\]
	
	Let $\nu$ be the word over $k$ symbols $s_1,\dots s_k$ defined by
	\[
		\nu(s_1,\dots,s_k) = w_k\left(
		A_i = \begin{cases}
			s_1, ~ i=0; \\
			s_i s_{i+1},~ i>0 \text{ even;} \\
			s_{i+1} s_i,~ i \text{ odd}. 
		\end{cases}
		, 0\leq i \leq k-1\right),
	\]
	which is positive since $w_k$ is positive (Definition \ref{def:words_w}) and evaluated on 
	positive words.
	Composing $\beta_{k,*}$ with the map $W_{k,*}$, which sends a suitable generator $c$ of 
	$\pi_1(\S^1)$ to the word $w_k(a_i)$, shows that $\theta_{k,*}$ satisfies
	\[
		\theta_{k,*}(c)= \beta_{k,*}\circ W_{k,*}(c) = \beta_{k,*}(w_k(a_i))=\nu(c_1,\dots,c_k).
	\]
	Positivity of $\nu$ implies that the map $\theta_{k,*}$ is injective, so the loop $\theta_k$  
	is nontrivial in $\pi_1(\cone)$. 
\end{proof}

From this we can get nontrivial loops in our group of interest $G$. Recall that by Corollary 
\ref{cor:words_induce_loops_in_cones}, if $k\geq 0$ is such that $\Rexp$ is $k$-step nilpotent, the 
path $\widetilde{\theta}_k$ in $\cone[G]$ is a loop that lifts $\theta_k$ through $p_{\omega}$. 
Since $\theta_k$ is nontrivial in $\pi_1(\cone)$, we get the following. 

\begin{Cor}
	For $k\geq 0$ such that the exponential radical of $G$ is $k$-step nilpotent, the lift 
	$\widetilde{\theta}_{k}$ is a 
	nontrivial loop in $\pi_1(\cone[G])$.
\end{Cor}

Finally, we get that $\cone[G]$ is not simply connected, for \emph{every} nonprincipal ultrafilter 
$\omega$.

\section{Changing the scale} \label{sec:change_scale}
In this section, we fix $k$ such that $\Rexp$ is $k$-step nilpotent, and denote
\[
	\theta=\theta_k=\lim_{\omega}\theta_{k,n}; \quad 
	\widetilde{\theta}=\widetilde{\theta}_k=\lim_{\omega} \widetilde{\theta}_{k,n},
\]
the path in $\cone$ and its lift in $\cone[G]$ given by Definition \ref{def:paths_theta}, which are 
both loops by Corollary \ref{cor:words_induce_loops_in_cones}.

We now iterate the construction of the previous sections while changing the scale. 
Let $\ell\geq 1$ be an integer. We define $\theta^{(\ell)}$ to be the loop in $\cone$ obtained by 
the 
process of evaluating the word $\Theta_{k}$ and taking the ultralimit, 
but replacing every occurrence of $t^{n}$ by $t^{n/\ell}$. Fractional powers are well defined here, 
since $t^{a}$ simply denotes the element $(0,0,a)$ in $\SOL$.

Put simply, we scale all loops by a factor of $1/\ell$.
This yields a loop of length roughly $1/\ell$ inside $\cone$, and similarly for its lift 
$\widetilde{\theta}^{(\ell)}$ in $\cone[G]$.
The loops $\theta^{(\ell)}$ for $l\geq 1$ accumulate at the point $\1_{\omega}$ at which they are 
all based. We aim to show that the union of their images is a subspace of $\cone$ homotopically 
equivalent to $\Earr$.

We follow the same outline as in Section \ref{sec:higher_order_loops} to highlight that the results 
and reasonings are similar, with the change-of-scale turning the circle into $\Earr$.

\subsection{Embedding the earring space into $\cone$}

We first introduce the subspace drawn by the loops $\theta^{(\ell)}$ in the 
asymptotic cone.

\begin{Def}
	We denote by $\SqEarr$, and call the \emph{square earrings}, the 
	subset of $\R^2$ defined as 
	the union of all squares of side length $\frac{1}{\ell}$ ($\ell\geq 1\in \N$), 
	with sides parallel to the axes and bottom-left corner at a given 
	basepoint; it is endowed with the natural graph metric.
	As a set, $\SqEarr = \bigcup_{l\ge 1}\partial ([0,1/\ell]^2).$
\end{Def} 

\begin{figure}
	\centering
	{\includegraphics[width=0.3\linewidth]{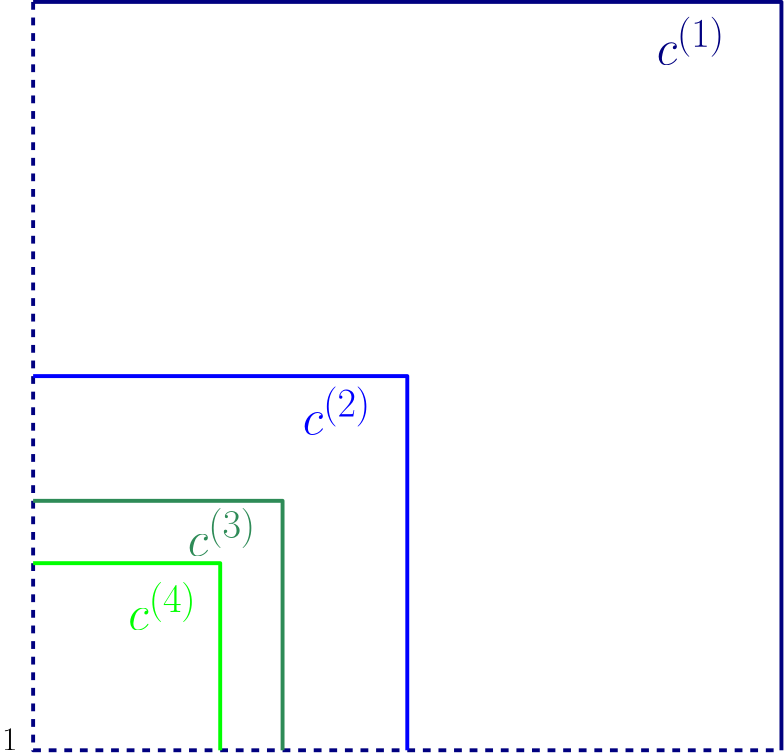}} 
	\qquad
	{\includegraphics[width=0.27\linewidth]{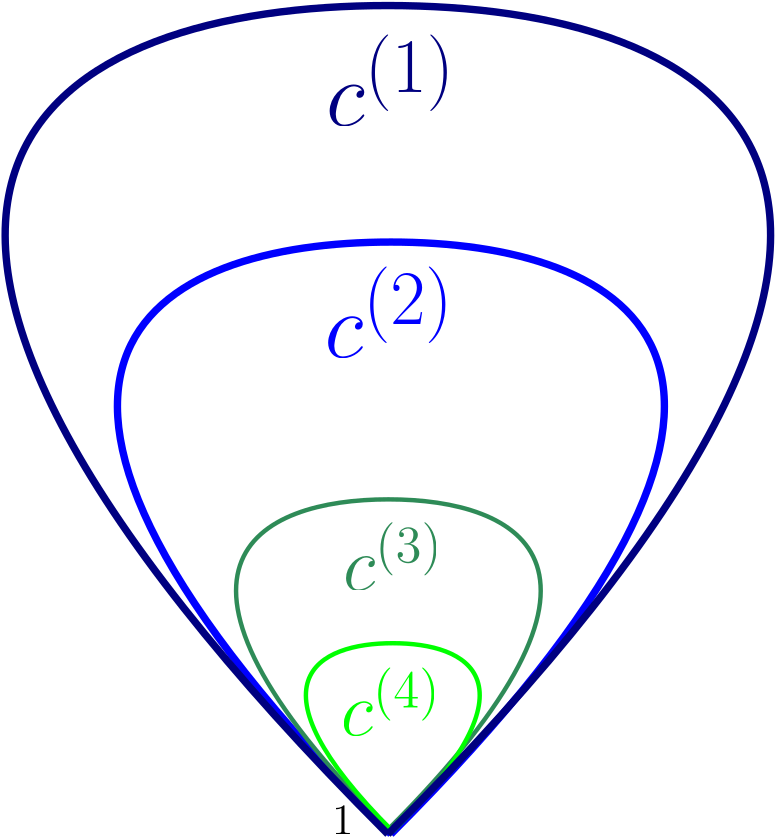}} 
	\caption{On the left, the square earring metric space $\SqEarr$. On the right, action 
	on the labels of an homotopy equivalence towards the usual earring space $\Earr$, contracting 
	the dashed sides.}
	\label{fig:SqEarr}
\end{figure}

See Figure \ref{fig:SqEarr} for a picture of $\SqEarr$.
This space is homotopically equivalent (but not homeomorphic) to $\Earr$, {via} a map 
contracting the bottom and left sides of every square onto the basepoint. 
Note that $\SqEarr$ can also be constructed as a quotient of $\Earr$ 
by gluing together the first and last quarters of consecutive circles.

We then have the following analog of Proposition 
\ref{prop:arcs_bj_embedding_circle} with the added change-of-scale, which is the key result leading 
to Theorem \ref{thm:technical_main}.

\begin{Prop} \label{prop:arcs_bj_change_of_scale_embedding_Earr}
	Let $i\neq i'$ be integers. The union of loops 
	$b_{i}^{(\ell)}\overline{b_{i'}^{(\ell)}}$ for $l\geq 1$ induces a bi-Lipschitz 
	embedding of $\SqEarr$ into $\cone$.
\end{Prop}

\begin{Rq}
	Notice the difference with Proposition \ref{prop:arcs_bj_embedding_circle}: 
	here the maps are not quasi-isometric embeddings 
	with uniform constants. Indeed, the {additive constants} found in the proof below depend 
	on the scale factor $\ell$. But crucially the 
	multiplicative constants do not and
	passing to asymptotic cones makes additive constants disappear,  so we still get bi-Lipschitz 
	embeddings between cones.
\end{Rq}

\begin{proof}
	First, recall from Proposition \ref{prop:dist_SOL} that we have the following estimate of 
	length in 
	$\SOL$, with constants depending only on the choice of metric:
	\[
	|(a,b,c)|_{\SOL} \simeq \log(1+|a|) + \log(1+|b|) + |c|.
	\]
	
	As in Section \ref{sec:higher_order_loops}, the point is to work with the evaluations in $\SOL$, meaning
	\[
		b_{j,n}^{(\ell)} = \lev B_j(U=t^{n/\ell}xt^{-(n/\ell)},Y=y) \rev[].
	\]
	 We apply Proposition \ref{prop:arcs_bj_embedding_circle} after substituting $n/\ell$ for $n$, 
	 which shows that the concatenation $b_{i}^{(\ell)}\overline{b_{i'}^{(\ell)}}$ is a 
	 quasi-isometric 
	 embedding of the circle of length $1/\ell$ in $\cone$ with constants independent of $n$ 
	 {and} $\ell$. 
	Hence the ultralimit loop $b_i^{(\ell)}\overline{b_{i'}^{(\ell)}}$ is a 
	bi-Lipschitz embedding of the circle of length $1/\ell$ in $\cone$ with 
	constants independent of $\ell$. 
	
	For $\ell \geq 1$, these loops are all based at the same point 
	$\1_{\omega}$ so they define a continuous map 
	\[
		\bigcup_{l \geq 1} b_i^{(\ell)}\overline{b_{i'}^{(\ell)}} : \Earr \to \cone
	\]
	that applies the $\ell$-th loop on the $\ell$-th circle of $\Earr$. 
	However, it is not injective: overlaps occur
	between the paths $b_j^{(\ell)}$ when $j \in \{i,i'\}$ and $\ell \geq 1$ vary. 
	When $j$, the exponent of $y$, is fixed, $b_j^{(\ell)}$ and $b_j^{(l')}$ share 
	their initial segment, 
	namely $\lim_{\omega} \lev  y^jt^{n/\ell'} \rev[\SOL]$ if $\ell \leq \ell'$; 
	and when the scale $\ell$ is fixed  all $b_j^{(\ell)}$'s share their endpoint 
	$z_l=\lim_{\omega}(t^{n/\ell}xt^{-(n/\ell)})$.
	
	Let us show that this is the only overlap occurring. More precisely, we claim 
	that the map defined above induces a bi-Lipschitz embedding into $\cone$ of the square 
	earrings space $\SqEarr$, endowed with the graph distance $d_g$.

	To achieve this, we use Proposition \ref{prop:dist_SOL} to compute 
	distances in $\cone$ between paths $b_j^{(\ell)}$ and $b_{j'}^{(\ell')}$ with 
	$j\leq j'$ and $\ell \leq \ell'$. We can assume that the scale factors satisfy $\ell< \ell'$, 
	the case 
	$\ell=\ell'$ without change-of-scale being already handled by Proposition 
	\ref{prop:arcs_bj_embedding_circle}. 
	Up to rescaling by $\ell$ both in $\SOL$ and in $\SqEarr$, we can also replace $\ell$ by 
	$1$ and all occurences of $1/\ell'$ by $\lambda =\ell/\ell'< 1$.
	
	Let us denote in this proof $\delta=b_j^{(1)}$ and $\delta'=b_{j'}^{(1/\lambda)}$, 
	and, for $n \geq 1$,
	\begin{align*}
		\delta_n= b_{j,n}^{(1)} = \lev y^j t^n x t^{-n} \rev[]; \qquad 
		\delta'_n=b_{j',n}^{(1/\lambda)} = \lev y^{j'} t^{\lambda n} x t^{-\lambda n} \rev[]
	\end{align*}
	the paths in $\SOL$ they are the ultralimit of. The path $\delta_n$ 
	consists, up to 
	parts of bounded length, of two segments $q$ and $r$ of length roughly $n$ 
	(``bounded'' and ``roughly'' depending here only on $j,j'$ and the choice 
	of metric, but not on $n$ and $\lambda$.) 
	Similarly, denote by $q'$ and $r'$ the 
	two segments of length roughly $\lambda n$ composing $\delta_n'$. 
	
	As in the proof of Proposition \ref{prop:arcs_bj_embedding_circle}, we consider points 
	$p$ on $\delta$ and $p'$ on $\delta_n'$ and compare their graph distance, when viewed on 
	$\SqEarr$, to their distance in $\SOL$. 
	
	\smallskip 
	\textbf{First case: $j=j'$ (same conjugate, changing the scale).} 
	
\begin{minipage}{.47\textwidth}
	Up to left translation by $y^j$, we can assume $y=0$. 
	The paths $\delta_n$ and $\delta'_n$ both share their initial segment 
	$q'=\lev t^{\lambda n} \rev[\SOL]$ which has length roughly $\lambda n$. They draw an 
	\texttt{F}-shaped subspace of $\SqEarr$, see Figure \ref{fig:arcs_in_sqearr}.
	If $p' \in q'$ then both $p$ and $p'$ lie on $\delta_n$ and the 
	estimate follows from Proposition \ref{prop:arcs_bj_embedding_circle}. 
	So assume $p'$ lies on the second segment $r'$, 
	say $p'=t^{\lambda n}x t^{ -k'}$ with $0 \leq k' \leq \lambda n$.
	We distinguish cases depending on the position of $p$.
\end{minipage}
\hfill
\begin{minipage}{.52\textwidth}
\centering
\includegraphics[width=.75\linewidth]{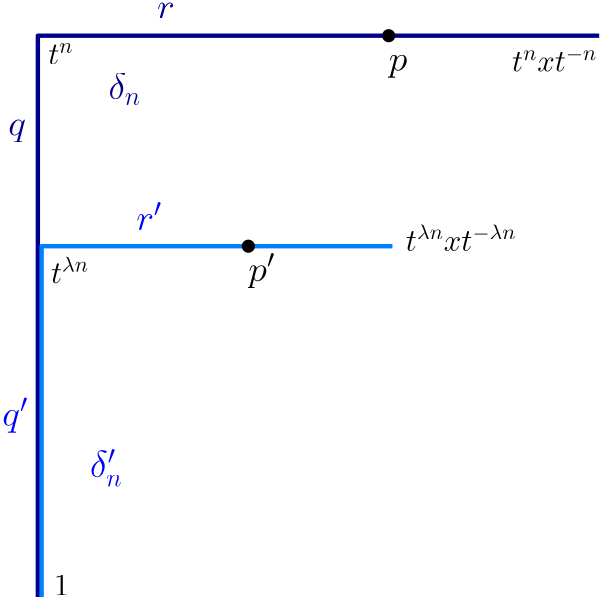}
\captionof{figure}{The paths $\delta_n$ and $\delta_n'$ inside $\SqEarr$ when $j=j'$. \\}
\label{fig:arcs_in_sqearr}

\end{minipage}

	\emph{(i) When $p$ lies on $q$.} If it lies on $q' \subset 
	q$, the computation follows from Proposition 
	\ref{prop:arcs_bj_embedding_circle} again, after substituting $\lambda n$ for $n$.
	So take $p=t^{\lambda n + k }$ lying on $q\setminus q'$, where $0\leq k \leq n(1-\lambda)$. 
	The graph distance between $p$ and $p'$ on $\SqEarr$ is roughly $k + 
	k'$. On the other hand, their distance in $\SOL$ is 
	\[
		d_{\SOL}(p,p') = d(t^{\lambda n +k}, t^{\lambda n}x t^{- k'}) = | 
		t^{-k} x t^{- k'} | \simeq k+ k',
	\]
	where the $t$-coordinate is responsible for roughly all the distance to the origin.

	\emph{(ii) When $p$ lies on $r$.} 
	Let now $p=t^nxt^{-k}$ lie on $r$, with $0\leq k \leq n$. The graph 
	distance between $p$ and $p'$ is $d_g(p,p') \simeq k + n(1-\lambda) + k'$, while the 
	distance in $\SOL$ is
	\begin{align*}
		d_{\SOL}(p,p') &= d(t^nxt^{-k}, t^{\lambda n} xt^{- k'})  \\
			&= | t^{ k'} x^{-1} t^{n-\lambda n} x t^{-k} |
			= |(x^{-1})^{( k')} x^{(n(1-\lambda )+ k')} t^{n(1-\lambda)+k'-k}|,
	\end{align*}
	where, if $z \in \SOL$, we denote by $z^{(u)}$ the conjugate of $z$ by $t^u$. 
	Going back to coordinates on $\SOL=\R^2\rtimes_{(1,-\alpha)} \R$ yields
	\begin{align*}
		d_{\SOL}(p,p') &= 
		\left|
			\left(e^{n(1-\lambda)+k'}-e^{k'}, 0, n(1-\lambda)+k'-k\right) 
		\right| \\
		&\simeq \log\left(1+e^{n(1-\lambda)+k'}-e^{k'}\right) 
		+ |n(1-\lambda)+k'-k|.
	\end{align*}
	The first term can be rewritten as 
	\[
		\log\left(1+e^{n(1-\lambda)+k'}-e^{k'}\right)
		=  \underbrace{n(1-\lambda)+k'}_{=d_g(p,p') - k} 
		+ \underbrace{\log\left(1 - e^{-n(1-\lambda)} 
			+ e^{-n(1-\lambda)-k'}\right).}_{R(\lambda,n,k')}
	\]
	By the estimate $u+|u-v| \simeq u+v$ 
	with $u=d_g-k$ and $v=k$, we have
	\[
		d(p,p') \simeq d_g - k + |(d_g-k)-k| + R(\lambda,n,k')
			\simeq d_g + R(\lambda,n,k'),
	\]
	where $d_g$ is the graph distance that we aim for
	and $R(\lambda,n,k')$ is a remainder that we wish to control.
	To do this, we introduce some dependence in $\lambda$. 
	See first that
	\[
		\log\left(1-e^{-n(1-\lambda)}+ e^{-n}\right)
		\leq 
		R(\lambda,n,k') = \log\left(1 - e^{-n(1-\lambda)} 
		+ e^{-n(1-\lambda)-k'}\right)
		\leq 0,
	\]
	 using the fact that $0\leq k'\leq \lambda n$. 
	 Now for fixed $\lambda$<1, the lower bound tends to $0$ when $n$ tends to 
	 infinity, so we can write 
	 $
	 	-O_{\lambda}(1) \leq R(\lambda,n,k') \leq 0,
	 $
	 where $O_{\lambda}(1)$ denotes some positive constant depending only on 
	 $\lambda$.
	Going back to distances in $\SOL$, the previous cases imply that there 
	exists a 
	constant $C_1>0$, uniform in all parameters, such that 
	\[
		\frac{1}{C_1}d_g(p,p') - O_{\lambda}(1) \leq d_{\SOL}(p,p') \leq C_1 d_g(p,p') + C_1,
	\]
	for all $p \in  \delta_n$ and $p' \in \delta_n'$. 
	\smallskip
	
	\textbf{Second case: $j<j'$ (changing both the conjugate and the scale).} 
	We treat the case $j=0$ and $j'=1$, the others being similar. Now $\delta_n$ and $\delta_n'$ 
	are on 
	opposite sides of $\SqEarr$ and do not intersect outside the basepoint. Again we split 
	cases according to the position of $p$.
	
	\emph{(i) When $p$ lies on $q$.} Write $p=t^k$ where $0\leq k \leq n$.
	Assume $p'$ lies on $q'$, say $p'=yt^{ k'}$ where $0\leq k'\leq \lambda n$. Then the 
	graph distance is roughly $k+k'$, while in $\SOL$
	\[
		d_{\SOL}(p,p') = | t^{-k} y t^{ k'} | = |y^{(-k)} t^{ k' -k} | 
		\simeq k + | k' - k| \simeq k+k'.
	\]
	If now $p'=yt^{\lambda n}x t^{ -k'}$ lies on $r'$, 
	then the graph distance is 
	$
		d_g \simeq k+ \lambda n + k',
	$
	while
	\[
		d_{\SOL}(p,p') = |t^{-k}yt^{\lambda n}x t^{-k'} | 
		= \left|y^{(-k)} x^{(\lambda n - k)}t^{\lambda n- k'-k} \right|.
	\]
	The element $y^{(-k)}=t^{-k}yt^k$ has length roughly $k$ in $\SOL$, 
	while $x^{(\lambda n - k)}$ has length roughly $\lambda n -k$ when 
	$\lambda n > k$ and is negligible otherwise. 
	Therefore, we have that
	\begin{align*}
		d(p,p') &\simeq 
		\underbrace{k + \max(\lambda n -k, 0 )}_{=\max(\lambda n, k) \simeq k + \lambda n} + 
		|\lambda n - k' - k| \\
		&\simeq k + \lambda n + |\lambda n - (k+ k') | \simeq 2k+ \lambda n + k',
	\end{align*}
	and this is indeed roughly $d_g$. 
	\smallskip
	
	\emph{(ii) When $p$ lies on $r$.} Write $p=t^nxt^{-k}$ with $0\leq k \leq n$.
	If $p'$ lies on $q'$ then in particular it lies on the translate path $y\delta_n$, 
	so no change-of-scale is needed and the 
	estimate follows from Proposition \ref{prop:arcs_bj_embedding_circle}.
	
	Finally, assume $p'=yt^{\lambda n}x t^{- k'}$ lies on $r'$ so that $d_g$ is 
	roughly $n$.
	In $\SOL$ we have
	\begin{align*}
		p^{-1}p' & =  
		(x^{-1})^{(k)} y^{(-n+k)} x^{(k+n(\lambda-1))} t^{(\lambda-1)n +k-k'} \\
		&= \left(-e^k(1-e^{n(\lambda-1)}), e^{\alpha(n-k)},(\lambda-1)n +k-k' \right),
	\end{align*}
	where in the last line we went back to coordinates. 
	Hence we get that 
	\[
		d(p,p') \simeq 
		\log \left(1 - e^{n(1-\lambda)+k} + e^k  \right)
		+n-k + |n(1-\lambda) + k' -k|. 
	\]
	The last term, coming from the $t$-coordinate, is roughly between $0$ and $n$. So factoring out 
	$e^k$ inside of the logarithm yields that
	$
		d(p,p') \simeq n + R'(\lambda,n,k)
	$
	with uniform constants, where $R'$ is a remainder which satisfies
	\[
		\log \left(1 - e^{-n(1-\lambda)}\right)
		\leq 
		R'(\lambda,n,k) = \log \left(1 - e^{-n(1-\lambda)} + e^{-k}\right) 
		\leq \log(2).
	\]
	As in the first case, for fixed $\lambda<1$ the lower bound tends to $0$ when $n$ tends to 
	infinity. 
	So we can have $ -O_{\lambda}(1) \leq R'(\lambda,n,k) \leq O_{\lambda}(1)$ with 
	$O_{\lambda}(1)$ denoting some positive constant depending only on $\lambda$. 
	
	The conclusion of all cases is that there exists a uniform constant $C>0$ such that, for $p \in 
	\delta_n$ and $p' \in \delta_n'$,
	\[
		\frac{1}{C}d_g(p,p') - O_{\lambda}(1) \leq d_{\SOL}(p,p') 
		\leq Cd_g(p,p') + O_{\lambda}(1).
	\]
	To go to the the asymptotic cone, rescale by $1/n$ and let $n$ tend to infinity along $\omega$. 
	This kills the additive constants, so that, for $p 
	\in \delta$ and $p' \in \delta'$, we have the following:
	\[
		\frac{1}{C}d_g(p,p') \leq d_{\omega}(p,p') \leq  C d_g(p,p').
	\]
	Now two points on the union
	$\bigcup_{l \geq 1} b_i^{(\ell)}\overline{b_{i'}^{(\ell)}}$
	lie on some $b_j^{(\ell)} \cup b_{j'}^{(\ell')}$, so this comparison shows that it 
	defines a bi-Lipschitz embedding of $\SqEarr$ as claimed.
\end{proof}

\begin{Rq}[On the dependence of constants]
	In this proof, note that $0<\lambda=\ell/\ell'<1$, but $\lambda$ can tend equally to 
	$0$ or to $1$ depending on $\ell$ and $\ell'$, hence the need to control the remainders $R$ and 
	$R'$ with functions of $\lambda$.
	Indeed, see that if $n$ and $k'$ are fixed while $\lambda$ tends to $1$ then $R(\lambda,n,k')$ 
	tends to $-k'$; so it cannot be bounded independently of all parameters.
	
	This is the manifestation of the worst-case scenario for the comparison between the graph 
	metric and the $\SOL$ metric. Namely, the endpoints of the paths $\delta_n$ and $\delta_n'$ 
	converge to the same point when 
	$\lambda$ tends to $1$ and $n$ is fixed, making their distance in $\SOL$ tend to $0$ while 
	their graph distance becomes roughly $2n$. But as we saw in the proof, this is solved by 
	letting first $n$ tend to infinity before varying the scale factor.
\end{Rq}

\subsection{Image of the rescaled loops $\theta^{(\ell)}$}
Now, we study all the loops $(\theta^{(\ell)})_{l \geq 0}$ together to get the following analog of 
Proposition \ref{prop:image_of_theta_in_cone_is_a_wedge}. 
Recall that $\nu$ is the positive word over $k$ letters from Corollary \ref{cor:theta_k_SOL}.

\begin{Prop} \label{prop:image_of_rescaled_loops}
	Let $\Phi : \Earr \to \cone$ be the map that applies the loop $\theta^{(\ell)}$ on the 
	$\ell$-th circle of $\Earr$. Then $\Phi$ is Lipschitz-continuous and has image homotopically 
	equivalent to $\Earr^{\vee k},$ a wedge of $k$ copies of $\Earr$ based at the singular point.
	
	More precisely, $\Phi$ factors as a composition $\beta \circ W$ 
	through $\Earr^{\vee k}$, where	
	\begin{itemize}
		\item $\beta : \Earr^{\vee k} \to \cone $ is a bi-Lipschitz embedding 
		(precomposed with a homotopy equivalence);
		\item $W: \Earr \to \Earr^{\vee k}$ is the continuous map 
		represented on the $\ell$-th circle $c^{(\ell)}$ by the word
		$
		 \nu(c_1^{(\ell)},\dots, c_k^{(\ell)}),
		$
		where $c_i^{(\ell)}$ denotes a suitable choice of orientation of the 
		$\ell$-th circle in the $i$-th copy  of $\Earr$, for $1\leq i \leq k$ and $1 \leq \ell$.
	\end{itemize} 
\end{Prop}

\begin{Rq}
	The wedge of $k$ copies of $\Earr$ based at the singular point is a countable union of 
	circles accumulating to a point, so it is in fact homeomorphic to $\Earr$. 
	This will be needed to obtain $\pi_1$-injectivity, but note that the map 
	$\Phi$ is not a homeomorphism, since the loop $\theta$ overlaps on itself. 
	This is the same phenomenon as in Proposition \ref{prop:image_of_theta_in_cone_is_a_wedge}, 
	and we do the same kind of factorization to get back to an embedding.
\end{Rq}
\begin{proof}
	The point is to put together Proposition \ref{prop:arcs_bj_change_of_scale_embedding_Earr} that 
	deals with changing the scale, and Proposition \ref{prop:image_of_theta_in_cone_is_a_wedge} 
	that tackles higher order commutators. 
	
	First, by Proposition \ref{prop:image_of_theta_in_cone_is_a_wedge}, the image in $\cone$ of 
	each loop $\theta^{(\ell)}=\theta_k^{(\ell)}$, for fixed $\ell \geq 1$, is 
	bi-Lipschitz-equivalent, with constants independent of $\ell$, to the rescaled cocoa pod 
	$\frac{1}{l}\Cocoa$ with arcs the paths $b_0^{(\ell)}, \dots, b_k^{(\ell)}$. In 
	particular, 
	the image of $\Phi$ is equal to the union of the images of the paths $b_i^{(\ell)}$ for $0\leq 
	i 
	\leq k$ and  $l\geq 0$.
	
	Next, by Proposition \ref{prop:arcs_bj_change_of_scale_embedding_Earr}, for fixed $i\neq i'$ 
	the loops $b_i^{(\ell)}\overline{b_{i'}^{(\ell)}}$ form a bi-Lipschitz copy of $\SqEarr$ inside 
	$\cone$ when $\ell$ varies. The Lipschitz constants depend on $i$ and $i'$, but in $\Phi$ only 
	$0\leq i,i' \leq k$ appear hence we get uniform constants.
	
	\begin{figure}
		\centering
		{\includegraphics[width=0.45\linewidth]{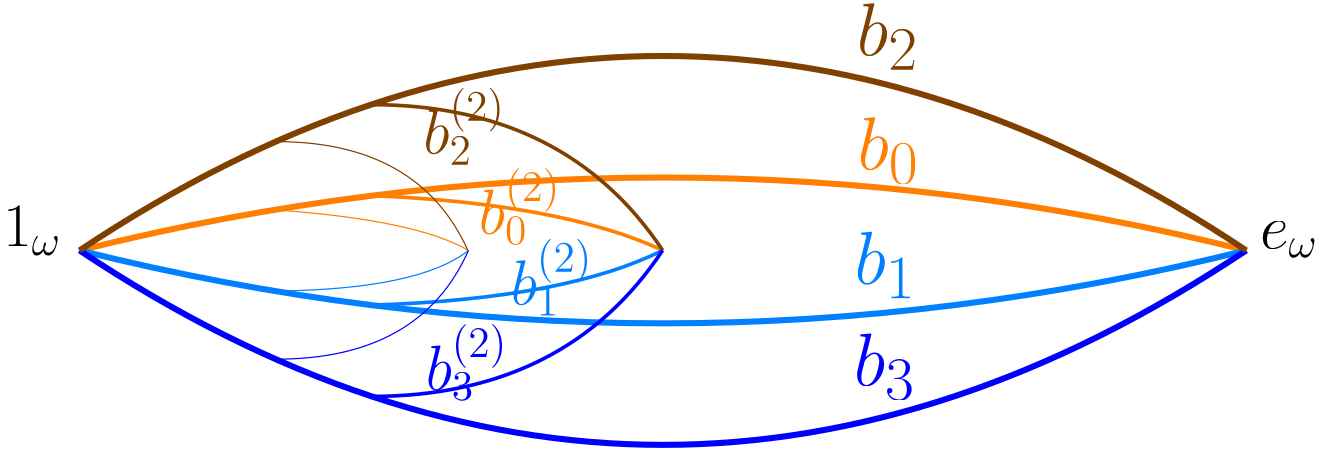}} 
		\quad
		{\includegraphics[width=0.23\linewidth]{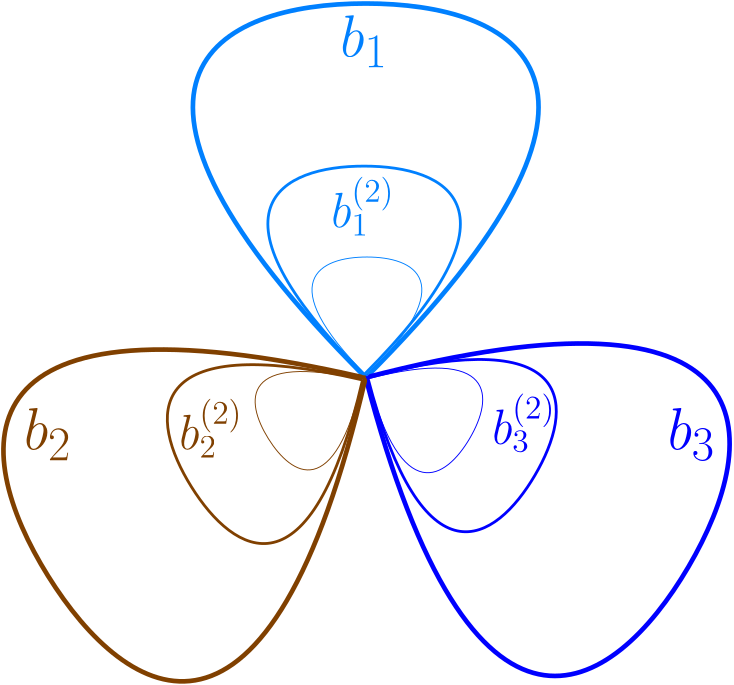}} 
		\captionof{figure}{On the left, union of the paths $b_{i}^{(\ell)}$ in $\cone$ for $0\leq i 
			\leq 3$ and $l\geq 0$, forming a square earring of cocoa pods with $4$ strands. On the 
			right, 
			action on the labels of an homotopy equivalence towards a wedge of $3$ earrings.}
		\label{fig:wedge_of_cocoa_pods}.
	\end{figure}	
	
	Together, these two facts yield that the image of $\Phi$ is bi-Lipschitz equivalent to a 
	``earring space of cocoa pods'', as in Figure \ref{fig:wedge_of_cocoa_pods}. 
	Putting together the homotopy equivalences between $\Cocoa$ and $(\S^1)^{\vee k}$ 
	(obtained by quotienting the segment $b_0$ to the basepoint) on the one hand, and between 
	$\SqEarr$ and $\Earr$ (obtained by quotienting the left and bottom sides of the 
	largest square to the basepoint)  on the other hand, shows that $\Phi(\Earr)$ is homotopically 
	equivalent to a wedge of $k$ copies of $\Earr$. 
	Moreover, we can choose this homotopy equivalence $f$ so 
	that the image of 
	$(b_i^{(\ell)}\overbar{b_0^{(\ell)}})^{(-1)^i}$ 
	inside $\pi_1(\Earr^{\vee k})$
	is a generator $c_i^{(\ell)}$ 
	of the $\ell$-th circle in the $i$-th copy of $\Earr$; 
	see Figure \ref{fig:wedge_of_cocoa_pods}.
	
	The map $\Phi$ itself is not an embedding, so we take the overlap into account and 
	factor it through $\Earr^{\vee k}$. As in the proof of Proposition 
	\ref{prop:image_of_theta_in_cone_is_a_wedge}, 
	we have the following commutative diagram.
	\[
	\begin{tikzcd}
		\Earr &&& \cone \\
		& {\Earr^{\vee k}} \\
		& {\Phi(\Earr)} 
		\arrow["{\Phi = \bigcup_{l\geq 0}\theta^{(\ell)}}", from=1-1, to=1-4]
		\arrow["{W}"', from=1-1, to=2-2]
		\arrow["{\beta}"', hook, from=2-2, to=1-4]
		\arrow["{f}","\sim" {rotate=90, anchor=north}, from=3-2, to=2-2]
	\end{tikzcd}
	\]
	Denote by $\beta$ the map that applies the 
	loop $(b_i^{(\ell)}\overline{b_0^{(\ell)}})^{(-1)^i}$ 
	on the circle $c_i^{(\ell)}$, so that $\beta$ is the composition of the bi-Lipschitz embedding 
	$\Phi(\Earr) \hookrightarrow \cone$ induced by the paths $b_i^{(\ell)}$ with a homotopic 
	inverse of the map $f:\Phi(\Earr) \to 
	\Earr^{\vee k}$.
	Now, using Proposition \ref{prop:image_of_theta_in_cone_is_a_wedge}, 
	$\Phi$ factors as $\Phi =  \beta \circ W$,
	where we see from Corollary \ref{cor:theta_k_SOL} 
	that $W$ is the map defined by sending $c^{(\ell)}$, the 
	$\ell$-th circle of $\Earr$, 
	to the positive word $\nu(c_1^{(\ell)}, \dots, c_k^{(\ell)})$.
\end{proof}

Recall that the loop $\theta^{(\ell)}$ in $\cone$ lifts, through the continuous surjection 
$p_{\omega} : \cone[G] \to \cone$, as a loop $\widetilde{\theta}^{(\ell)}$ in $\cone[G]$. 
Denote by $ \widetilde{\Phi} : \Earr \to \cone[G]$ the continuous map defined by applying
$\widetilde{\theta^{(\ell)}}$ on the $\ell$-th circle of $\Earr$. 
Then $\widetilde{\Phi}$ lifts $\Phi$ through $p_{\omega}$, meaning that 
$p_{\omega} \circ \widetilde{\Phi} =\Phi$.

\begin{Cor} \label{cor:Phi_pi_1_injective}
	The map $\Phi : \Earr \to \cone $ is $\pi_1$-injective. Hence so is its lift 
	$\widetilde{\Phi} : \Earr \to \cone[G]$. 
\end{Cor}
As before, we obtain $\pi_1$-injectivity for both the geometric $\beta$ and the combinatorial $W$. 
The geometric part is dealt with using covering dimension. The combinatorial part uses first the 
positivity of the image of generators, and 
second the fact that there is no ``mixing of scales'' in the action of $\Phi$, meaning that it 
sends the generator 
of scale $1/\ell$ to a word over only the generators of scale $1/\ell$.

\begin{proof}
We make use of the factorization $\Phi=\beta \circ W$ from Proposition 
\ref{prop:image_of_rescaled_loops}.
Omitting homotopy equivalences, 
the map  $\beta$ is a bi-Lipschitz embedding of $\Earr^{\vee k}$ inside 
the asymptotic cone $\cone$.
Since $\Earr^{\vee k}$ is homeomorphic to $\Earr$ 
and homotopy copies of $\Earr$ are retracts up to homotopy inside spaces of covering dimension $1$ 
(Corollary \ref{cor:embedded_covdim_one}) 
such as $\cone$ (Theorem \ref{thm:dim_cone_SOL}), 
the map $\beta$ is 
$\pi_1$-injective.

So it is enough to check that the map 
$
	{W} : \Earr \to \Earr^{\vee k}
$ 
is $\pi_1$-injective to complete the proof.
For this, we make use of the fact that $\pi_1(\Earr)$ embeds as a subgroup of the projective limit of the free groups 
$F_n=\langle c^{(1)}, \dots, c^{(n)}  \rangle$
; see Theorem \ref{thm:description_pi_1_Earr}.
Fix a homeomorphism from $\Earr^{\vee k}$ to $\Earr$ that sends $c_i^{(\ell)}$, the 
$\ell$-th circle of the $i$-th copy of $\Earr$, to $c^{((\ell-1)k+i)}$. Then we get from the 
definition of $W$ in Proposition \ref{prop:image_of_theta_in_cone_is_a_wedge} that $W_*$ is the 
restriction to $\pi_1(\Earr)$ of the endomorphism
of $\varprojlim_{n \geq 1} F_n$ induced by the family of homomorphisms
\begin{align*}
	&\Omega_n : F_n \longrightarrow F_{nk} \\
	&c^{(\ell)} \mapsto \nu(c^{((\ell-1)k+1)}, \dots, c^{(\ell k)}).
\end{align*}

We claim that $\Omega_n$ is injective for all $n\geq 1$. First, the image of each generator is a 
\emph{positive} word in $F_{nk}$ since $\nu$ is positive, hence it is non-trivial.
Moreover, for $1 \leq \ell \leq n$, $\Omega_n$ sends the generator $c^{(\ell)}$ to  a word where 
only generators 
$c^{((\ell-1)k+1)}, \dots , c^{(\ell k)}$ occur -- this comes from the fact that $W$ sends the 
$\ell$-th circle of $\Earr$ to a word over only $\ell$-th circles of copies of $\Earr$.
As the intervals $[(\ell-1)k+1, \ell k]$ are disjoint when $\ell$ varies, the sets of 
generators appearing in the words $\Omega_n(c^{(l)})$ are 
disjoint when $l$ varies.
This implies that $\Omega_n$ sends reduced words to reduced words, hence it is injective.

Since this holds for all $n\geq 1$, the map $W_*$ is injective.
Coming back to $\Phi$, it factors as the composition of $\beta$ and $W$ which are both 
$\pi_1$-injective, so it is as well. Hence so is its lift $\widetilde{\Phi}$. 
\end{proof}

This concludes the proof of Theorem \ref{thm:technical_main}, hence of Theorem \ref{thm:main}. It 
only remains to deal with homology, with the added difficulty that we know no explicit description 
of the elements of $\Ho_1(\Earr)$ as there was for $\pi_1(\Earr)$. 

\begin{Prop}
	The map $\Phi_*: \Ho_1(\Earr) \to \Ho_1(\cone)$ is injective when restricted to a subgroup 
	isomorphic to $\Z^\N$. Hence so is $\widetilde{\Phi} : \Ho_1(\Earr) \to \Ho_1(\cone[G])$. 
\end{Prop}

\begin{proof}
	Recall that $\Ho_1(\Earr)$ is the abelianization of $\pi_1(\Earr)$; denote by $e_{\ell}$ the 
	image 
	of the generator  $c^{(\ell)}$ in $\Ho_1(\Earr)$.
	The inclusion $\pi_1(\Earr) < \varprojlim F_n$ induces a surjective (but not injective) 
	morphism $\sigma : \Ho_1(\Earr) \to \Z^\N$ 
	which sends $e_{\ell}$ to the $\ell$-th basis vector in $\Z^\N$. 
	This map admits a section $s: \Z^\n \to \Ho_1(\Earr)$ defined by 
	\[
		s : (x_i)_{i \in {\N^*}} \mapsto 
		(e_1^{x_1} \dots e_n^{x_n})_{n\geq 1} \mod{[\pi_1(\Earr),\pi_1(\Earr)]},
	\]
	where the sequence $(e_1^{x_1} \dots e_n^{x_n})_{n\geq 1}$ denotes an element 
	of $\pi_1(\Earr)$ viewed inside $\varprojlim_{n}F_n$.
	
	We identify $\Z^\N$ with its image through $s$, and show that the map $\Phi_*$ is 
	injective in restriction to this subgroup. 
	Using Proposition \ref{prop:image_of_rescaled_loops} write $\Phi =  \beta \circ W$ where 
	$\beta$ is a bi-Lipschitz embedding of $\Earr^{\vee k}$ (omitting homotopy equivalences) and 
	$W$ sends $c^{(l)}$ to the loop $\nu(c_1^{(\ell)}, \dots, c_k^{(\ell)})$. 
	As before, homotopy copies of $\Earr$ are retracts up to homotopy inside spaces of covering 
	dimension $1$ (Corollary \ref{cor:embedded_covdim_one}), 
	so the map $\beta$ induces an injection in homology. 
	
	We claim that $W_*$ is injective on $\Z^\N$. Using positivity of 
	the word $\nu$ we see that $W_*:\Ho_1(\Earr) \to \Ho_1(\Earr^{\vee k})$ 
	sends, for $\ell \geq 1$, the element $e^{(\ell)}$ to  
	$\sum_{i=1}^k  v_{i,\ell} e^{(\ell -1)k+i}$, 
	where $v_{i,\ell}$ is a nonnegative integer and $(v_{1,\ell},\dots, v_{k,\ell})$ is nonzero. In 
	restriction to $\Z^N$ we may view $W_*$ as an infinite matrix indexed by $\N\times \N$, 
	where all columns are nonzero and have disjoint support. Indeed, the $\ell$-th column may have 
	nonzero coefficients only in lines $[(\ell-1).k+1,\dots, \ell.k]$. 
	So ${W_*}_{|\Z^\N}$ is injective.
	
	Hence $\Phi_* = \beta_* \circ W_*$ is injective on $\Z^\N$, 
	and so is its lift $\widetilde{\Phi}_*$.
	\end{proof}

This result yields the first part of Theorem \ref{thm:main_homology}. For homology over $\Q$, 
see that $\Z^\N$ is torsion-free and has continuum cardinality, so $\Z^\N \otimes \Q$ is a 
vector subspace of $\Ho_1(\Earr,\Q)$ of continuum dimension over $\Q$.
This concludes the proof of Theorem \ref{thm:main_homology}.

\bibliographystyle{amsalpha}
\bibliography{Bib-Pi1(Cone)_SOL.bib}

\end{document}